\pdfoutput=1

\documentclass[english]{jnsao}
\usepackage[capitalise]{cleveref}
\usepackage[justification=centering]{subcaption}
\usepackage{pgfplots}
\pgfplotsset{compat=1.16}
\usepgfplotslibrary{fillbetween}

\newlength\fheightscen \newlength\fwidthscen 
\newlength\fheight \newlength\fwidth 

\newcommand{\N}{\mathbb{N}}

\newcommand{\R}{\mathbb{R}}
\newcommand{\Normal}{\mathcal{N}}
\newcommand{\X}{\mathcal{X}}
\newcommand{\Y}{\mathcal{Y}}
\newcommand{\V}{\mathcal{V}}
\newcommand{\Z}{\mathbb{Z}}
\newcommand{\E}{\mathbb{E}}
\newcommand{\Fourier}{\mathcal{F}}
\newcommand{\Fper}{\mathcal{F}_P^\text{per}}
\newcommand{\di}{\mathrm{d}}
\newcommand{\dprime}{{\prime\prime}}
\newcommand{\udag}{{u^\dagger}}
\newcommand{\Psimap}{\Psi_\mathrm{MAP}}
\newcommand{\Phimap}{\Phi_\mathrm{MAP}}
\newcommand{\Id}{\mathrm{Id}}

\renewcommand{\epsilon}{\varepsilon}
\renewcommand{\phi}{\varphi}
\renewcommand{\P}{\mathbb{P}}

\newcommand{\norm}[1]{{\lVert #1 \rVert}}
\newcommand{\bignorm}[1]{\left\lVert #1 \right\rVert}
\newcommand{\abs}[1]{{\lvert #1 \rvert}}
\newcommand{\bigabs}[1]{\left\lvert #1 \right\rvert}
\newcommand{\scalprod}[1]{{\langle #1 \rangle}}
\newcommand{\bigscalprod}[1]{{\left\langle #1 \right\rangle}}
\newcommand{\dualpair}[1]{{\langle #1 \rangle}}
\newcommand{\bigdualpair}[1]{{\left\langle #1 \right\rangle}}
\newcommand{\Prob}[1]{\mathbb{P}\left[ #1 \right]}
\newcommand{\Exp}[1]{\mathbb{E}\left[ #1 \right]}

\DeclareMathOperator{\supp}{supp}
\DeclareMathOperator{\tr}{tr}
\DeclareMathOperator{\ran}{ran}

\DeclareMathOperator{\spn}{span}

\DeclareMathOperator{\erfc}{erfc}
\DeclareMathOperator{\Var}{Var}
\DeclareMathOperator{\esssupp}{ess\,supp}

\newcommand{\rev}[1]{#1} 

\theoremstyle{plain}
\newtheorem{assumption}{Assumption}

\crefname{assumption}{Assumption}{Assumptions}
\crefname{remark}{Remark}{Remarks}
\crefname{example}{Example}{Examples}
\crefname{lemma}{Lemma}{Lemmas}
\crefname{theorem}{Theorem}{Theorems}
\crefname{corollary}{Corollary}{Corollaries}

\title{Maximum a posteriori testing in \\statistical inverse problems}
\shorttitle{MAP testing in statistical inverse problems}
\author{%
	Remo Kretschmann\thanks{Institute of Mathematics, University of Potsdam, Karl-Liebknecht-Straße 24--25, 14476 Potsdam, Germany\\(\email{remo.kretschmann@uni-potsdam.de})}
	\and Frank Werner\thanks{Institute of Mathematics, University of Würzburg, Emil-Fischer-Straße 30, 97074 Würzburg, Germany\\(\email{frank.werner@uni-wuerzburg.de})}
}
\shortauthor{R.~Kretschmann, F.~Werner}
\date{2025-03-24}

\hypersetup{
	pdfkeywords = {inverse problem} {hypothesis testing} {ill-posed problem} {Bayesian inference} {deconvolution} {backward heat equation},
}

\begin{document}

\maketitle

\begin{abstract}
This paper is concerned with a Bayesian approach to testing hypotheses in statistical inverse problems. Based on the posterior distribution $\Pi \left(\cdot |Y = y\right)$, we want to infer whether a feature $\langle\varphi, u^\dagger\rangle$ of the unknown quantity of interest $u^\dagger$ is positive. This can be done by the so-called maximum a posteriori test. We provide a frequentistic analysis of this test's properties such as level and power, and prove that it is a regularized test in the sense of Kretschmann et al. (2024). Furthermore we provide lower bounds for its power under classical spectral source conditions in case of Gaussian priors. Numerical simulations illustrate its superior performance both in moderately and severely ill-posed situations.
\end{abstract}

\section{Introduction}

\subsection{Linear-Gaussian set-up}

We consider a linear inverse problem
\begin{equation}
	\label{model1}
	Y = Tu^\dagger + \sigma Z
\end{equation}
where $T$ is a bounded linear operator between two separable Hilbert spaces $\X$ and $\Y$, $u^\dagger$ is an unknown quantity of interest, $\sigma > 0$ a noise level, and $Z$ is a Gaussian white noise process on $\Y$.
Both $Z$ and $Y$ are Hilbert space processes on $\Y$, i.e., equation \eqref{model1} has to be understood as
\begin{equation*}
	Y(g) = \scalprod{Tu^\dagger,g}_\Y + \sigma Z(g) \quad\text{for all}~g \in \Y.
\end{equation*}
Here, the Hilbert space processes $Y,Z$: $\Y \to L^2(\Omega,\mathcal{F},\mu)$ arise from the family $(Z(g))_{g\in\Y}$, $(Z(g))_{g\in\Y}$ of Gaussian random variables $Z(g) \sim \Normal(0,\norm{g}_\Y^2)$ with covariance $\E[Z(g)Z(h)] = \int_\Omega Z(g)Z(h) \di\mu = \scalprod{g,h}_{\Y}$. Note, that for fixed $\omega \in \Omega$, $g \mapsto Z(g,\omega) = Z(g)(\omega)$ is a linear functional on $\Y$, but in general not continuous. That is, $\omega \mapsto Z(\cdot,\omega)$ is a random variable with values in the \emph{algebraic} dual space $\Y^*$ of $\Y$. The same holds true for the whole data $Y$.

If we embed $\Y = \Y^0$ into a scale of Hilbert spaces $\{\Y^r\}_{r \in \R}$ defined by an orthonormal basis $(\phi_k)_{k\in\N}$ of $Y$ and a positive $\ell^2$-sequence $(a_k)_{k\in\N}$ as
\[ \Y^r := \left\{y \in \Y: \sum_{k=1}^\infty a_k^{-2r}\abs{\scalprod{y,\phi_k}_Y}^2 < \infty\right\} \]
for $r \ge 0$, and
\[ \Y^{-r} = \left\{y \in L(\Y,\R): \sum_{k=1}^\infty a_k^{2r}\abs{y(\phi_k)}^2 < \infty\right\} \]
for $r > 0$, then $Z$ takes a value in the \emph{topological} dual space $\mathcal{Z}' = \Y^{-1}$ of $\mathcal{Z} := \Y^{1}$ almost surely, and $\Y$ itself is the Cameron--Martin space of its law, as is described in appendix 7.4 of \cite{Nickl2020Bernstein}. For this reason, we can interpret \eqref{model1} rigorously as an equation in $\mathcal{Z}'$.

Models of the form \eqref{model1} are widely used as prototypes in applications ranging from astrophysics to cell biology, as central limit theorems and asymptotic equivalence statements motivate the Gaussianity of the noise in \eqref{model1}.

\subsection{Estimation and inference}

Due to the wide applicability, estimation of $u^\dagger$ in \eqref{model1} based on the available data $Y$ has been tackled extensively in the literature. We refer to \cite{ehn96, bhmr07} for filter-based methods, to \cite{js91,mr96} for methods based on the singular value decomposition of $T$, and to \cite{d95,as98} for wavelet-based methods. 

However, the focus of this paper is not on the reconstruction of all of $u^\dagger$, but on inference about specific features of $u^\dagger$ that can be described via linear functionals. This is of special interest in applications, as often (in the example of functions) only properties such as modes, monotonicity, or the support of $u^\dagger$ are of interest. 

We will follow the general idea of \cite{KWW:2023} and tackle this issue as a hypothesis testing problem. Therefore, let $\phi \in \X^*$ be a bounded linear functional and $\dualpair{\phi,u^\dagger}$ the \emph{feature of interest}. We want to test 
\begin{equation}\label{test}
H:\dualpair{\phi,u^\dagger} \le 0\qquad\text{vs.}\qquad K: \dualpair{\phi,u^\dagger} > 0,
\end{equation}
and denote $\X_H = \{x \in \X: \dualpair{\phi,x} \le 0\}$ and $\X_K = \{x \in \X: \dualpair{\phi,x} > 0\}$.

Testing hypotheses in statistical inverse problems has so far mostly been considered for global testing problems of the form
\begin{equation}\label{test2}
H_0: u^\dagger = 0 \qquad\text{vs.}\qquad H_1 : u^\dagger \in B, \left\Vert u^\dagger\right\Vert_{\X} > \rho
\end{equation}
with smoothness classes $B \subseteq \X$ and radii $\rho >0$. This situation has, e.g., been investigated in \cite{llm12,iss12,ilm14,mm14}, and in those works, sharp minimax bounds for the corresponding detection thresholds are provided. However, many local features of $u^\dagger$ cannot be described by a global testing problem such as \eqref{test2}. Consequently, also local testing problems \eqref{test} have been discussed in the literature. An early example is \cite{shmd13}, where (multiscale) inference for features in deconvolution models has been investigated. More recently, also more general models have been considered, see e.g. \cite{ebd17,pwm18,depsh19}. In the most recent work \cite{KWW:2023}, a regularized approach to testing local features has been discussed, which resolves two essential problems in previous works. The regularized approach allows to all test features $\phi \in \X^*$, and not only regular features $\phi \in R \left(T^*\right)$. Furthermore, the finite sample properties of regularized tests are way superior to previous approaches. The aim of the paper at hand is to further explore this regularized approach and to investigate it from a Bayesian point of view.
\rev{In particular, we are interested in studying how Tikhonov--Phillips regularized tests arise through Bayesian modelling with a Gaussian prior distribution.}

\subsection{Bayesian treatment}

During the last decade, and inspired by the monograph \cite{ks05} and the survey \cite{s10}, the Bayesian approach to the inverse problem \eqref{model1} has become popular. 
This means, that the unknown, now denoted by $U$, is no longer assumed to be a fixed, deterministic quantity, but a prior distribution $\Pi$ is assigned to it and in what follows, $U \sim \Pi$ is considered random as well. In this paper, we generically consider
\begin{equation}\label{eq:model_bayes}
	 Y = TU + \sigma Z,
	  \end{equation}
and assign a Gaussian prior distribution $\Pi = \Normal(m_0,C_0)$ to $U$. Here, $C_0$ is a symmetric, positive definite, trace class operator, and $U$ and $Z$ are assumed to be independent. We can interpret $U$ as Gaussian process on $\X$ by considering $i_\X \circ U$, where $i_\X$: $\X \to \X'$ denotes the Riesz map $x \mapsto \scalprod{x,\cdot}_\X$, see \cite{Lehtinen1989}.

To clarify the notation, let $\P_u$ for given $u \in \X$ denote the distribution of $Y$ on $\Y^*$, given $U = u$. Then, we interpret $H$ and $K$ as subsets of the class $\{\P_u: u \in \X\}$. Let $\nu$ denote the distribution of $\sigma Z$ on $\Y^*$. The distribution $\P_u = \nu(\cdot - Tu)$ of $Y$, given $U = u$, is absolutely continuous with respect to $\nu$ for all $u \in \X$ with
\[ L(u,Y) := \frac{\di\nu(\cdot - Tu)}{\di\nu}(Y) = \exp \left(\frac{1}{\sigma^2}Y(Tu) - \frac{1}{2\sigma^2}\norm{Tu}_\Y^2\right) \quad \text{a.s.,} \]
see appendix 7.4 of \cite{Nickl2020Bernstein}.
For $y \in \Y^*$, the posterior distribution $\Pi(\cdot|Y = y) = \Normal(m,C)$ on $\X$ is almost surely Gaussian with mean
\begin{equation}
	\label{eq:posterior_mean}
	\begin{aligned}
		m &= m_0 + C_0^\frac12 \left(C_0^\frac12 T^*T C_0^\frac12 + \sigma^2\Id\right)^{-1} C_0^\frac12 T^* \left(y - Tm_0\right) \\
		&= C_0^\frac12 \left(C_0^\frac12 T^*T C_0^\frac12 + \sigma^2\Id\right)^{-1} C_0^\frac12 T^*y
		+ \left[\Id - C_0^\frac12 \left(C_0^\frac12 T^*T C_0^\frac12 + \sigma^2\Id\right)^{-1} C_0^\frac12 T^*T\right] m_0
	\end{aligned}
\end{equation}
and covariance
\[ C = \sigma^2 C_0^\frac12 \left(C_0^\frac12 T^*T C_0^\frac12 + \sigma^2\Id\right)^{-1} C_0^\frac12,
\]
and $(B,y) \mapsto \Pi(B,Y = y)$ is a regular conditional distribution of $U$, given $Y$.
This follows from \cite[Thm~4.2]{Lehtinen1989}, as described in section 2 of \cite{Agapiou2018}.
\rev{As mentioned above, the choice of a Gaussian prior distribution is, i.a., motivated by the expected connection between tests defined in terms of the resulting posterior distribution and Tikhonov--Phillips regularized tests.}
The measure $\Pi(\cdot|Y = y)$ is almost surely absolutely continuous w.r.t. $\Pi$ according to Bayes' theorem with Radon-Nikodym derivative
\[ \frac{\di\Pi(\cdot|Y = y)}{\di\Pi}(U) = \frac{L(U,y)}{\int_\X L(x,y) \di\Pi(x)}. \]

Theoretical properties of the posterior $\Pi(\cdot|Y = y)$ in Bayesian inverse problems have been studied widely in the literature. We refer to \cite{llm15,Agapiou2018} for its contraction properties in the linear case, and, e.g., to \cite{Nickl2020Bernstein} for a frequentistic analysis of its coverage in a nonlinear problem. Besides this, also non-Gaussian distributions have been treated, e.g., in \cite{l12a,l12b,lss09,ls04}.

\subsection{Our contribution and structure of the paper}

The major contribution of the paper at hand is to develop a Bayesian approach for testing problems of the form \eqref{test}. In the following Section 2, we will introduce maximum a posteriori (MAP) tests, which naturally arise from the posterior $\Pi(\cdot|Y = y)$. Within the considered setup, this distribution is Gaussian and hence we are able to explicitly analyze frequentistic properties of the corresponding test such as its level analytically. Interestingly, it turns out that without further a priori assumptions on $u^\dagger$, it is impossible to derive a non-trivial bound for the size of the MAP test (\cref{size_map_test}). This is in agreement with the findings from \cite{KWW:2023}, where such a bound was also obtained only under a priori assumptions. Furthermore, we show that the MAP test can be characterized as a regularized test in the sense of \cite{KWW:2023}, and show that optimal detection properties are almost reachable by choosing the prior covariance $C_0$ appropriately (\cref{inf_prior_inf_reg}). In Section 3, we will then pursue a different analytic approach and consider a priori assumptions in form of classical Hölder-type source conditions and provide bounds for the power of the MAP test under such conditions. Optimality of these bounds can be obtained for an a priori choice of the prior covariance, and we also discuss an a posteriori choice. In Section 4, we finally perform extensive numerical simulations to investigate the finite sample performance of the MAP test both in mildly ill-posed and in severely ill-posed situations. We end this paper with an outlook in Section 5.

\section{Maximum a posteriori tests}

Before we describe the approach to Bayesian testing, let us briefly recall the notation for hypothesis tests and the corresponding frequentistic properties in general. A (non-randomized) test is a measurable mapping $\Psi$, that maps the available data $Y$ either to $0$ or $1$. Given such a test $\Psi$, the quantity $\P_\udag\left[\Psi(Y) = 1\right]$ is called the \emph{size} of the test $\Psi$, where $\P_\udag$ denotes the law of $Y$ in \eqref{eq:model_bayes} given $U = u^\dagger$. 
The maximal size of the test under the hypothesis,
\[ \sup \left\{ \P_\udag\left[\Psi(Y) = 1\right] : \udag \in \X_H \right\} \]
is called \emph{level (of significance)} of the test $\Psi$, measuring the maximal probability of a false rejection (type 1 error). In frequentistic statistics, this quantity is controlled by a user-chosen value $\alpha \in (0,1)$. If $u^\dagger \in \X_K$, i.e., if we are working under the alternative, then the size of $\Psi$ is also called the \emph{power} of this test, and it measures the probability of a correct detection. Clearly, this quantity --- from a frequentistic perspective --- should be as large as possible. 

Given the Bayesian set-up \eqref{eq:model_bayes} and the derived posterior $\Pi \left(\cdot | Y = y\right)$, it seems clear how to construct a Bayesian test. Essentially, the posterior covers all information about the model (including prior and data), and from a Bayesian perspective the hypothesis should be accepted if and only if the set $\X_H$ has a has a higher probability under the posterior $\Pi \left(\cdot | Y = y\right)$ than the alternative $\X_K$. This leads directly to the definition of the \emph{maximum a posteriori (MAP) test} $\Psimap$ by
\[ \Psimap(y) := \begin{cases}
	1 & \text{if}~\Pi(\X_K|Y = y) > \Pi(\X_H|Y = y), \\
	0 & \text{otherwise}.
\end{cases} \]
For this test, $\Psimap(y) = 1$ if $\Pi(\X_K|Y = y) > \frac12$ and $\Psi(y) = 0$ if $\Pi(\X_K|Y = y) \le \frac12$.
In the context of statistical classification, $\Phimap$ is also called a \emph{Bayes classifier} since it minimizes the probability of misclassification from a Bayesian perspective.
In the following, we will study the MAP test $\Psimap$ and its properties \emph{from a frequentistic point of view}, i.e., under the assumption that it is applied to data $Y$ generated by a fixed truth $\udag \in \X$ according to \eqref{model1}.

\begin{remark}
	Here, it becomes clear why we have to consider the hypothesis $\scalprod{\phi,\udag} \le 0$ rather than $\scalprod{\phi,\udag} = 0$ when assigning a Gaussian prior --- the event $\X_H$ needs to have positive prior probability in order to lead to a nontrivial MAP test.
\end{remark}

\begin{remark}
	MAP tests are in general not connected to MAP estimates. While a set of maximal posterior probability is chosen, the mode of the posterior distribution does not play a role in this choice. Instead, the whole posterior distribution has to be taken into account.
\end{remark}

As the following theorem shows, the MAP test can be evaluated in a simple manner.
\begin{theorem}
	\label{form_MAP_test}
	For $y \in \Y^*$, it holds that
	\[ \Psimap(y) = \begin{cases}
		1 & \text{if}~\dualpair{y,\Phimap} > \scalprod{m_0, T^*\Phimap - \phi}_\X, \\
		0 & \text{otherwise},
	\end{cases} \]
	where
	\begin{equation}
		\label{eq:def_Phimap}	
		\Phimap := TC_0^\frac12 \left(C_0^\frac12 T^*T C_0^\frac12 + \sigma^2 \Id\right)^{-1} C_0^\frac12 \phi.
	\end{equation}
\end{theorem}
\begin{proof}
First of all, we express the MAP test in terms of the law $\Pi_\phi(\cdot|Y = y)$ of $\dualpair{\phi, U}$, given $Y = y$, the marginal posterior distribution
\[ \Pi_\phi(\cdot|Y = y) = \Normal(\dualpair{\phi,m},\dualpair{\phi,C\phi}), \]
since $\Pi(\X_K|Y = y) = \Pi_\phi((0,\infty)|Y = y)$.
The cdf $F_\phi$ of $\Pi_\phi(\cdot|Y = y)$ is given by
\[ F_\phi(t) = \Prob{\dualpair{\phi,U} \le t \vert Y = y} = Q\left(\frac{t - \dualpair{\phi,m}}{\dualpair{\phi,C\phi}^{1/2}}\right), \]
where $Q$ denotes the cdf of the standard normal distribution. Now, we have
\begin{equation}
	\label{eq:eval_map_test}
	\begin{aligned}
		\Pi(\X_K|Y = y) > \frac12 \quad&\Leftrightarrow\quad
		\Prob{\dualpair{\phi,U} > 0 \vert Y = y} > \frac12 \quad\Leftrightarrow\quad 
		F_\phi(0) < \frac12 \\
		&\Leftrightarrow\quad \dualpair{\phi,m} > 0.
	\end{aligned}
\end{equation}
\rev{Now, \eqref{eq:posterior_mean} and the definition of $\Phimap$ allow us to express $\scalprod{m,\phi}_\X$ as}
\begin{equation}
	\label{eq:eval_mean}
	\begin{aligned}
		\scalprod{m,\phi}_\X &= \bigdualpair{y, TC_0^\frac12 \left(C_0^\frac12 T^*T C_0^\frac12 + \sigma^2 \Id\right)^{-1} C_0^\frac12 \phi} \\
		&\quad + \bigscalprod{m_0, \left[\Id - T^*T C_0^\frac12 \left(C_0^\frac12 T^*T C_0^\frac12 + \sigma^2 \Id\right)^{-1} C_0^\frac12\right] \phi}_\X \\
		&= \dualpair{y,\Phimap} - \scalprod{m_0, T^*\Phimap - \phi}_\X,
	\end{aligned}
\end{equation}
which finishes the proof.
\end{proof}

\rev{In \cref{form_MAP_test},} we see that the MAP test rejects if the linear estimator $\dualpair{Y,\Phimap}$ exceeds the critical value $\dualpair{m_0,T^*\Phimap - \phi}$ and note that the probe element $\Phimap$ is independent of the prior mean $m_0$.
\rev{Equation \eqref{eq:eval_map_test}, moreover, tells us that the MAP test rejects if the posterior mean displays the feature of interest. In the linear-Gaussian setting, the posterior mean is equal to the posterior mode (or MAP estimate) and thus corresponds to the minimizer of a Tikhonov--Phillips functional. This leads to the following characterization of the probe element $\Phimap$.}

\begin{theorem}
	\label{TP_reg_sol}
	\rev{Assume} that $T$ is a compact, linear operator and that $C_0$ and $T^*T$ commute.
	In this case, $\Phimap$ is the minimizer of the functional
	\begin{equation}
		\label{eq:Tikhonov_Phillips_functional}
		\Phi \mapsto \norm{T^*\Phi - \phi}_\X^2 + \sigma^2\bignorm{C_0^{-\frac12}V^*\Phi}_\Y^2
	\end{equation}
	in $\ran (VC_0^{1/2})$, \rev{where $V$: $\X \to \Y$ is a unitary operator such that $T = V\abs{T}$.}
\end{theorem}
\begin{proof}
\rev{By the assumptions on $T$ and $C_0$,} there exists a common system of eigenvectors $(e_k)_{k\in\N}$ of $T^*T$ and $C_0$ that forms an orthonormal basis of $\overline{\ran T^*} = (\ker T)^\perp$. This yields a singular value decomposition
\[ Tx = \sum_{k\in\N} \tau_k \scalprod{x,e_k}_\X f_k, \quad x \in X, \]
of $T$ with an orthonormal system $(f_k)_{k\in\N}$ in $\Y$ and $\tau_k > 0$ for all $k \in \N$, as well as a spectral decomposition
\[ C_0 x = \sum_{k\in\N} \rho_k \scalprod{x,e_k}_\X e_k, \quad x \in (\ker T)^\perp, \]
of the restriction of $C_0$ to $(\ker T)^\perp$ with $\rho_k > 0$ for all $k \in \N$.
\rev{Moreover, the operator $V$ admits the representation}
\[ Vx = \sum_{k\in\N} \scalprod{x,e_k}_\X f_k \quad \text{for all}~x \in \X. \]
Note that $(\ker T)^\perp$ is invariant both under $C_0^{1/2}$ and $C_0^{1/2}T^*TC_0^{1/2} + \sigma^2\Id$.
This leads to
\begin{equation}
	\label{eq:spectral_Phimap}
	\begin{aligned}
		\Phimap &= TP_{(\ker T)^\perp}C_0^\frac12\left(C_0^\frac12 T^*T C_0^\frac12 + \sigma^2\Id\right)^{-1}C_0^\frac12\phi \\
		&= TC_0^\frac12\left(C_0^\frac12 T^*T C_0^\frac12 + \sigma^2\Id\right)^{-1}C_0^\frac12P_{(\ker T)^\perp}\phi \\
		&= \sum_{k\in\N} \frac{\tau_k\rho_k}{\tau_k^2\rho_k + \sigma^2} \scalprod{\phi,e_k} f_k
		= \sum_{k\in\N} \frac{1}{\tau_k^2 + \sigma^2\rho_k^{-1}} \scalprod{T\phi,f_k} f_k \\
		&= V\left(T^*T + \sigma^2C_0^{-1}\right)^{-1}V^*T\phi,
	\end{aligned}
\end{equation}
where $C_0^{-1}$ only exists as an unbounded linear operator in $\X$ and is not defined on the whole space.
It follows that
\begin{equation*}
	\left(TT^* + \sigma^2 VC_0^{-1}V^*\right)\Phimap = V\left(T^*T + \sigma^2C_0^{-1}\right)V^*\Phimap = T\phi,
\end{equation*}
which implies that $\Phimap$ is the minimizer of \rev{\eqref{eq:Tikhonov_Phillips_functional}.}
\end{proof}

\rev{\Cref{TP_reg_sol} shows that the probe element $\Phimap$ is a Tikhonov-Phillips regularized solution to the ill-posed equation $T^*\Phi_0 = \phi$ with weighted quadratic norm penalty. Here, the ratio of $\sigma^2$ to the size of the prior covariance $C_0$ plays the role of a regularization parameter.}

\begin{theorem}
	\label{size_map_test}
	For $\udag \in \X$, the size of the test $\Psimap$ is given by
	\begin{equation}
		\label{eq:size_map_test}
		\P_\udag\left[\Psimap(Y) = 1\right] = Q\left(\frac{\scalprod{\udag,T^*\Phimap}_\X - \scalprod{m_0,T^*\Phimap - \phi}_\X}{\sigma\norm{\Phimap}_\Y}\right)
	\end{equation}
	with $\Phimap$ as in \eqref{eq:def_Phimap}, where $Q$ denotes the cdf of the standard normal distribution.
\end{theorem}
\begin{proof}
Using the model \eqref{model1}, \eqref{eq:eval_map_test}, \eqref{eq:eval_mean}, and $\dualpair{Z,\Phimap} \sim \Normal(0, \norm{\Phimap}_\X^2)$, we obtain
\begin{align*} 
	\P_\udag[\Psimap(Y) = 1]
	&= \P_\udag[\dualpair{Y,\Phimap} - \dualpair{m_0,T^*\Phimap - \phi} > 0] \\
	&= \Prob{\dualpair{T\udag,\Phimap} + \sigma\dualpair{Z,\Phimap} - \dualpair{m_0,T^*\Phimap - \phi} > 0} \\
	&= \Prob{\dualpair{Z,\Phimap} > \frac{\dualpair{m_0,T^*\Phimap - \phi} - \dualpair{u,T^*\Phimap}}{\sigma}} \\
	&= 1 - Q\left(\frac{\dualpair{m_0,T^*\Phimap - \phi} - \dualpair{u,T^*\Phimap}}{\sigma\norm{\Phimap}}\right) \\
	&= Q\left(\frac{\dualpair{u,T^*\Phimap} - \dualpair{m_0,T^*\Phimap - \phi}}{\sigma\norm{\Phimap}}\right). \qedhere
\end{align*}
\end{proof}


\subsection{Necessity of a priori assumptions}

In this section, we discuss the case that nothing is known about the truth other than that $\udag \in \X$ and establish necessary and sufficient conditions for a MAP test to have a nontrivial level of significance.
We also relate the MAP test to unregularized hypothesis tests discussed in section 1.2 of \cite{KWW:2023}.
First, we observe that in order to obtain a MAP test whose level is smaller than $1$, $\phi$ typically needs to be an eigenvector of $C_0^{1/2}T^*TC_0^{1/2}$.

\begin{theorem}
	\label{not_EV_level_1}
	\rev{Assume that $C_0$ and $T^*T$ commute.}
	If $\phi$ is not an eigenvector of $C_0^{1/2} T^*T C_0^{1/2}$, then $\Psimap$ has level
	\[ \sup \left\{ \P_\udag\left[\Psimap(Y) = 1\right] : \udag \in \X_H \right\} = 1. \]
\end{theorem}
\begin{proof}
We prove this statement by contraposition.
First of all, we compute
\begin{equation*}
	\begin{aligned}
		T^*\Phimap &= T^*T C_0^\frac12 \left(C_0^\frac12 T^*T C_0^\frac12 + \sigma^2 \Id\right)^{-1} C_0^\frac12 \phi 
		= C_0^\frac12 T^*T C_0^\frac12 \left(C_0^\frac12 T^*T C_0^\frac12 + \sigma^2 \Id\right)^{-1} \phi \\
		&= \left[\Id - \sigma^2\left(C_0^\frac12 T^*T C_0^\frac12 + \sigma^2 \Id\right)^{-1}\right] \phi
		= \left[\Id - \left(\Id + \frac{1}{\sigma^2} C_0^\frac12 T^*T C_0^\frac12\right)^{-1}\right] \phi,
	\end{aligned}
\end{equation*}
where we used that $C_0^{1/2}$ commutes with $T^*T$.
Let us assume that the test $\Psimap$ has a level of significance smaller than $1$. Then, the term 
\[ \scalprod{T^*\Phimap, \udag}_\X = \bigscalprod{\Id - \left(\Id + \sigma^{-2}C_0^{\frac12}T^*TC_0^\frac12\right)^{-1}\phi, \udag}_\X \]
in \eqref{eq:size_map_test} is uniformly bounded from above for all $\udag \in \X_H$.
In particular, the subspace
\[ W := \{u \in \X: \scalprod{\phi,u} = 0\} \]
is invariant under the self-adjoint operator 
\[ A := \Id - \left(\Id + \frac{1}{\sigma^2}C_0^\frac12 T^*T C_0^\frac12\right)^{-1}, \]
because if there existed $w \in W$ with $\scalprod{T^*\Phimap,w} = \scalprod{A\phi,w} = \scalprod{\phi, Aw} \neq 0$, then $t w \in W \subset \X_H$ as well for all $t \in \R$ and $\scalprod{T^*\Phimap,t w} = t\scalprod{\phi,Aw}$ could be arbitrarily large.
Since $\scalprod{A\phi,w} = \scalprod{\phi,Aw} = 0$ for all $w \in W$ we have $A\phi \in W^\perp = \spn \{\phi\}$, i.e. $\phi$ is an eigenvector of $A$. Therefore, $\phi$ is also an eigenvector of $C_0^{1/2}T^*TC_0^{1/2} = \sigma^2((\Id - A)^{-1} - \Id)$.
\end{proof}

The following theorem shows that the condition that $\phi$ is an eigenvector of $C_0^{1/2}T^*TC_0^{1/2}$ is rather restrictive.

\begin{theorem}
	\label{EV_span}
	Let $T$ be an injective, compact, linear operator on $\X$ and assume that $C_0$ and $T^*T$ commute.
	Then any eigenvector $\phi$ of $C_0^{1/2} T^*T C_0^{1/2}$ is an element of $\spn \left\{e_k: k \in \N\right\}$, where $(e_k)_{k\in\N}$ is a joint eigenbasis of $C_0$ and $T^*T$.
	Moreover,
	\[ \left\{ x \in \X: \scalprod{x,e_k}_\X = 0~\text{for all}~k \in \N~\text{with}~\scalprod{\phi,e_k}_\X = 0 \right\} \]
	is an eigenspace of $C_0^{1/2}T^*TC_0^{1/2}$.
\end{theorem}
\begin{proof}
By assumption, $T^*T$ and $C_0$ are jointly diagonalizable with eigensystems $(\tau_k^2,e_k)_{k\in\N}$ an $(\rho_k,e_k)_{k\in\N}$, respectively, and $(e_k)_{k\in\N}$ forms an orthonormal basis of $\X$ due to the injectivity of $T$.
Let $\phi \in \X$ be an eigenvector of $C_0^{1/2}T^*TC_0^{1/2}$ associated with the eigenvalue $\gamma > 0$.
Then, we have
\[ \sum_{k=1}^\infty \tau_k^2 \rho_k \scalprod{e_k,\phi}_\X e_k = C_0^\frac12 T^*T C_0^\frac12\phi = \gamma\phi = \sum_{k=1}^\infty \gamma \scalprod{e_k,\phi}_\X e_k. \]
Comparing coefficients yields that $\tau_k^2\rho_k = \gamma$ for all $k \in \N$ with $\scalprod{e_k,\phi}_\X \neq 0$. Since both $\tau_k$ and $\rho_k$ tend to $0$ as $k \to \infty$, this implies that all but finitely many components $\scalprod{e_k,\phi}_\X$ must be zero.
Note that any $e_k$ with $\scalprod{\phi,e_k}_\X \neq 0$ is an eigenvector of $C_0^{1/2}T^*TC_0^{1/2}$ as well, and thus any $x \in \X$ that satisfies $\scalprod{x,e_k}_\X = 0$ for all $k \in \N$ with $\scalprod{\phi,e_k}_\X = 0$.
\end{proof}

\begin{corollary}
	\label{not_span_level_1}
	If $T$ is an injective, compact, linear operator and $C_0$ and $T^*T$ commute, then the MAP test $\Psimap$ has level $1$ for any $\phi \in \X \setminus \spn \{e_k: k \in \N\}$, where $(e_k)_{k\in\N}$ is a joint eigenbasis of $C_0$ and $T^*T$.
\end{corollary}
\begin{proof} 
	By \cref{EV_span}, $\phi$ is not an eigenvector of $C_0^{1/2}T^*TC_0^{1/2}$, so that $\Psimap$ has level $1$ by \cref{not_EV_level_1}.
\end{proof}

We see that without a priori assumptions on $\udag$, it can typically not be expected that the MAP test has a level $\alpha < 1$. If, on the other hand, $\phi$ is an eigenvector of $C_0^{1/2}T^*TC_0^{1/2}$, the situation is different.

\begin{theorem}
	\label{level_alpha_eigenvector}
	Assume that $C_0$ and $T^*T$ commute and let $\alpha \in (0,1)$. If $\phi \in \X$ is an eigenvector of $C_0^{1/2} T^*T C_0^{1/2}$ associated with the eigenvalue $\gamma > 0$ and the prior mean $m_0 \in \X$ satisfies
	\[ \scalprod{m_0,\phi}_\X = \frac{\sigma^2 + \gamma}{\sigma} \norm{\Phimap}_\Y Q^{-1}(\alpha), \]
	then the test $\Psimap$ has level of significance
	\[ \sup \left\{ \P_\udag\left[\Psimap(Y) = 1\right] : \udag \in \X_H \right\} = \alpha. \]
\end{theorem}
\begin{proof}
By assumption, we have $C_0^{1/2} T^*T C_0^{1/2} \phi = \gamma \phi$, which yields
\begin{equation*}
	T^*\Phimap = \left(\Id - \left(\Id + \sigma^{-2}C_0^\frac12 T^*T C_0^\frac12\right)^{-1}\right)\phi
	= \left(1 - \left(1 + \frac{\gamma}{\sigma^2}\right)^{-1}\right)\phi
	= \frac{\gamma}{\gamma + \sigma^2}\phi.
\end{equation*}
It follows that
\[ \scalprod{\udag,T^*\Phimap}_\X =  \frac{\gamma}{\gamma + \sigma^2}\scalprod{\udag,\phi}_\X \le 0 \quad \text{for all}~\udag \in \X_H \]
with equality for $\udag = 0$.
Moreover, we have
\[ T^*\Phimap - \phi = -\left(1 + \frac{\gamma}{\sigma^2}\right)^{-1}\phi = -\frac{\sigma^2}{\sigma^2 + \gamma}\phi, \]
which leads to
\[ \scalprod{m_0,T^*\Phimap - \phi}_\X = -\frac{\sigma^2}{\sigma^2 + \gamma}\scalprod{m_0,\phi}_\X = -\sigma\norm{\Phimap}Q^{-1}(\alpha). \]
Now, it follows from \cref{size_map_test} that
\[ \P_\udag[\Psimap(Y) = 1] = Q\left(\frac{\dualpair{u,T^*\Phimap} - \dualpair{m_0,T^*\Phimap - \phi}}{\sigma\norm{\Phimap}}\right) \le \alpha \]
for all $\udag \in \X_H$ with equality for $\udag = 0$.
\end{proof}

In case that $C_0$ and $T^*T$ commute, any eigenvector $\phi$ of $C_0^{1/2}T^*TC_0^{1/2} = T^*TC_0$ lies in $\ran T^*$.
This way, there exists $\Phi_0 \in \Y$ such that $T^*\Phi_0 = \phi$. We consider the unbiased frequentist estimator
\[ \scalprod{Y,\Phi_0} = \scalprod{Tu^\dagger + \sigma Z,\Phi_0} = \scalprod{u^\dagger,\phi} + \sigma\scalprod{Z,\Phi_0} \]
for $\scalprod{u^\dagger,\phi}$. It gives rise to the \emph{unregularized} test
\[ \Psi_0(Y) = \begin{cases}
	1 & \text{if}~\scalprod{Y,\Phi_0} \ge c, \\
	0 & \text{otherwise}.
\end{cases} \]
We obtain a level-$\alpha$ test by choosing $c := -\sigma\norm{\Phi_0}Q^{-1}(\alpha)$.
This way,
\begin{align*}
	P_u\{\scalprod{Y,\Phi_0} \ge c\} 
	&= P_u\left\{\scalprod{Z,\Phi_0} \ge \frac{c}{\sigma} - \frac{\scalprod{\udag,\phi}}{\sigma}\right\}
	= P_u\left\{\scalprod{Z,\Phi_0} \le \frac{\scalprod{\udag,\phi}}{\sigma} - \frac{c}{\sigma}\right\} \\
	&\le P_u\left\{\scalprod{Z,\Phi_0} \le -\frac{c}{\sigma}\right\}
	= Q\left(-\frac{c}{\sigma\norm{\Phi_0}}\right) = \alpha
\end{align*}
for all $\udag \in \X_H$ with equality for $\udag = 0$.
The MAP test $\Psimap$ is related to the unregularized test $\Phi_0$ as follows.

\begin{theorem}
	\label{map_unreg_equal}
	Let $T$ be an injective, compact, linear operator and assume that $C_0$ and $T^*T$ commute.
	If $\phi \in \X$ is an eigenvector of $C_0^{1/2}T^*TC_0^{1/2}$, then the level-$\alpha$ MAP test $\Psimap$ with $m_0$ chosen as in \cref{level_alpha_eigenvector} is equal to the unregularized level-$\alpha$ test $\Psi_0$ based upon the estimator $\scalprod{Y,\Phi_0}$, where $T^*\Phi_0 = \phi$. Moreover, both tests are uniformly most powerful.
\end{theorem}
\begin{proof}
First of all, we show that $\Phi_0$ is an eigenvector of $TC_0T^*$ associated with the eigenvalue $\gamma$, where $\gamma > 0$ is the eigenvalue associated with the eigenvector $\phi$ of $C_0^{1/2}T^*TC_0^{1/2}$.
By assumption, $T^*T$ and $C_0$ are jointly diagonalizable. Let $(\tau_k^2,e_k,f_k)_{k\in\N}$ denote a singular system of $T$ and $(\rho_k,e_k)_{k\in\N}$ an eigensystem of $C_0$.
By assumption, $\tau_k^2\rho_k = \gamma$ for all $k \in \N$ with $\scalprod{\phi,e_k}_\X \neq 0$.
Moreover, we have $\scalprod{\phi,e_k}_\X = \scalprod{T^*\Phi_0,e_k}_\X = \tau_k\scalprod{\Phi_0,f_k}_\Y$ and $\tau_k > 0$ for all $k \in \N$.
Therefore,
\[ \scalprod{TC_0T^*\Phi_0,f_k}_\Y = \tau_k^2\rho_k\scalprod{\Phi_0,f_k}_\Y = \gamma\scalprod{\Phi_0,f_k}_\Y \]
for all $k \in \N$ with $\scalprod{\Phi_0,f_k}_\Y \neq 0$, which implies that $TC_0T^*\Phi_0 = \gamma\Phi_0$.
Next, we compute
\[ \norm{\Phimap}_\Y = \norm{TC_0(T^*TC_0 + \sigma^2\Id)^{-1}\phi}_\Y = \frac{\gamma^\frac12}{\gamma + \sigma^2} \norm{C_0^{\frac12}\phi}_\X. \]
By the previous considerations,
\[ \scalprod{\phi,C_0\phi} = \scalprod{\Phi_0,TC_0T^*\Phi_0} = \gamma\scalprod{\Phi_0,\Phi_0}. \]
This yields $\norm{\Phimap}_\Y = \frac{\gamma}{\gamma + \sigma^2} \norm{\Phi_0}_\Y$. The estimator
\[ \dualpair{Y,\Phimap} = \scalprod{T\udag,\Phimap}_\Y + \sigma\dualpair{Z,\Phimap} = \scalprod{\udag,T^*\Phimap}_\X + \sigma\dualpair{Z,\Phimap} \]
has a Gaussian distribution with mean $\scalprod{\udag,T^*\Phimap}_\X = \frac{\gamma}{\gamma + \sigma^2}\scalprod{\udag,\phi}_\X$ and variance $\sigma^2\norm{\Phimap}_\Y^2$.
Both random variables 
\[ \scalprod{Y,\Phi_0} = \Normal(\scalprod{u^\dagger,\phi},\sigma^2\norm{\Phi_0}^2) 
\quad \text{and} \quad 
\scalprod{Y,\Phimap} \sim \Normal\left(\frac{\gamma}{\gamma + \sigma^2}\scalprod{u^\dagger,\phi},\sigma^2\norm{\Phimap}^2\right) \] 
have monotone likelihood ratio in $T_1(z) = \scalprod{z,\Phi_0}$ or $T_2(z) = \scalprod{z,\Phimap}$, respectively, when seeking the parameter $\scalprod{u^\dagger,\phi}$. Therefore, both tests $\Psi_0$ and $\Psi$ are uniformly most powerful level-$\alpha$ tests due to their form as indicator functions of superlevel sets of $T_{1/2}$ by Theorem 3.4.1 in \cite{LehRom:2005}. Both tests are in fact \emph{equal} since the random variables $\frac{\gamma}{\gamma + \sigma^2}\scalprod{Y,\Phi_0}$ and $\scalprod{Y,\Phimap}$ are equal in distribution, and the rejection thresholds of the two tests are chosen to achieve the same level of significance.
\end{proof}


\subsection{Interpretation as regularized test}\label{sec:reg}

In this section, we illustrate how MAP tests based upon a Gaussian prior fit into the framework of regularized hypothesis tests discussed in section 2 of \cite{KWW:2023} and how they relate to optimal regularized tests discussed in section 3 of \cite{KWW:2023}.
\rev{Let us therefore briefly recall the notation and central results from \cite{KWW:2023}.
\begin{definition}
	Under the notion \emph{regularized test}, we understand any test of the form
	\[ \Psi_{\Phi,c}(Y) := \begin{cases}
		1 & \text{if}~\scalprod{Y,\Phi} > c, \\
		0 & \text{else},
	\end{cases} \]
	with $\Phi \in \Y$ and $c \in \R$.
\end{definition}}

\rev{These tests are based upon linear estimators $\scalprod{Y,\Phi}$ for $\scalprod{\phi,\udag}_\X$.}
As we have seen in \cref{form_MAP_test}, the MAP test $\Psimap$ corresponds to the specific regularized test $\Psi_{\Phi_\text{MAP},c_\text{MAP}}$ with
\begin{align*}
	\Phi_\text{MAP} &= TC_0^\frac12 \left(C_0^\frac12 T^*T C_0^\frac12 + \sigma^2 \Id\right)^{-1} C_0^\frac12 \phi, \\
	c_\text{MAP} &:= \scalprod{m_0, T^*\Phimap - \phi}_\X.
\end{align*}

\rev{In order to be able to control the level of any given regularized test,} we make the following assumptions on $\X$, $\Y$, $\udag$, $T$, and $\phi$.
\begin{assumption}
	\label{ass:V}
	There exists a pair of Banach spaces $(\V,\V')$ such that the following holds.
	\begin{enumerate}
		\item \label{ass:CS} $\scalprod{u,v}_\X \le \norm{u}_{\V'} \norm{v}_\V$ for all $v \in \V \cap \X$ and $u \in \V' \cap \X$.
		\item \label{ass:udag_bounded} $\udag \in \V \cap \X$ with $\norm{\udag}_\V \le 1$.
		\item \label{ass:Ts_bounded} $\ran T^* \subseteq \V'$ and $T^*$: $\Y \to \V'$ is bounded.
		\item \label{ass:phi} $\phi \in \overline{\ran T^*}$, where the closure is formed in $\V'$.
	\end{enumerate}
\end{assumption}
\rev{\cref{ass:V}.\ref{ass:udag_bounded} can be interpreted as a spectral source condition if $\V$ is a subspace of smoother functions, but also allows for other situations of interest, such as if $\udag$ is a density or a bounded function. For further discussion of \cref{ass:V} we refer to section 2 of \cite{KWW:2023}.}
Note, that Theorem \ref{not_EV_level_1} does no longer apply under Assumption \ref{ass:V} since the set of admissible elements $\udag$ is now limited. 

\begin{theorem}
	\label{power_reg_level_alpha_test}
	\rev{Under \cref{ass:V} and for every $\Phi \in \Y \setminus \{0\}$,} the regularized test \rev{$\Psi_{\Phi,c(\Phi)}$} with
	\begin{equation}
		\label{eq:c_reg}
		c(\Phi) := \sigma \norm{\Phi}_\Y Q^{-1}(1 - \alpha) + \norm{T^*\Phi - \phi}_{\V'}
	\end{equation}
	has at most level $\alpha \in (0,1)$, i.e., $\sup \left\{ \P_\udag\left[\Psi_{\rev{\Phi,c(\Phi)}}(Y) = 1\right] : \udag \in \X_H \right\} \le \alpha$, and its power is given by
	\begin{equation}
		\label{eq:size_reg_test_alpha}
		\P_\udag[\Psi_{\rev{\Phi,c(\Phi)}}(Y) = 1] = Q\left(Q^{-1}(\alpha) - \frac{J_{T\udag}^\Y(\Phi)}{\sigma}\right)
	\end{equation}
	for every $\udag \in \X_K$, where
	\begin{equation}
		\label{eq:def_J}
		J_{T\udag}^\Y(\Phi) := \frac{\norm{T^*\Phi - \phi}_{\V'} - \scalprod{\Phi, T\udag}_\Y}{\norm{\Phi}_\Y}.
	\end{equation}
	\rev{Moreover, for every $u^\dagger \in \X_K$, there exists a minimizer $\Phi^\dagger$ of $J_{T\udag}^\Y$ in $\Y \setminus \{0\}$, and the test $\Psi_{\Phi^\dagger,c(\Phi^\dagger)}$ with $c$ as in \eqref{eq:c_reg} has maximal power within the class $\{\Psi_{\Phi,c(\Phi)}: \Phi \in \Y \setminus \{0\}\}$.}
\end{theorem}
\begin{proof}
	\rev{In \cite[Theorem 2.3]{KWW:2023}, it is proven that $\Psi_{\Phi,c}$ has at most level $\alpha$ for a choice of $c = c(\Phi)$ as in \eqref{eq:c_reg}. 
	Expression \eqref{eq:size_reg_test_alpha} for the size of $\Psi_{\Phi,c(\Phi)}$ is derived in equations (3.1) and (3.2) in \cite{KWW:2023}.
	The existence of an optimal regularized test $\Psi_{\Phi^\dagger,c(\Phi^\dagger)}$ is proven in \cite[Theorem 3.2]{KWW:2023}.}
\end{proof}

\begin{corollary}
	\rev{The MAP test $\Psimap$} has at most level $\alpha$ if the prior mean is chosen such that $c_\text{MAP} = c(\Phi_\text{MAP})$, i.e., such that
	\begin{equation}
		\label{cond_m0}
		\scalprod{m_0, T^*\Phimap - \phi}_\X = \sigma \norm{\Phimap}_\Y Q^{-1}(1 - \alpha) + \norm{T^*\Phimap - \phi}_{\V'}.
	\end{equation}
	With this choice of $m_0$, the power of $\Psimap$ \rev{for $\udag \in \X_K$} is, moreover, given by
	\begin{equation}
		\label{eq:size_map_test_alpha}
		\P_\udag[\Psimap(Y) = 1] = Q\left(Q^{-1}(\alpha) - \frac{J_{T\udag}^\Y(\Phimap)}{\sigma}\right).
	\end{equation}
\end{corollary}
\begin{proof}
	\rev{This follows immediately from \cref{power_reg_level_alpha_test} and \cref{size_map_test}, where} we used that $Q^{-1}(\alpha) = -Q^{-1}(1 - \alpha)$.
\end{proof}

In the following, we will always choose $m_0 \in \X$ such that \eqref{cond_m0} holds, \rev{i.e., we will consider a setting that results in a MAP test $\Psimap = \Psi_{\Phimap, c(\Phimap)}$ with frequentistic level at most $\alpha$.}
It is a natural question whether \rev{this test} can achieve \rev{the power of the optimal regularized test $\Psi_{\Phi^\dagger,c(\Phi^\dagger)}$.}
In the following, we show that MAP tests with arbitrarily close to optimal power can indeed be constructed by choosing the prior covariance accordingly. These are, however, oracle tests, i.e. their construction requires knowledge of the truth $\udag$.

\begin{theorem}
	\label{min_in_cl_ran_T}
	The minimizer $\Phi^\dagger$ of $J_{T\udag}^\Y$ satisfies $\Phi^\dagger \in \overline{\ran T}$.
\end{theorem}
\begin{proof}
	First, let us decompose $\Phi^\dagger = \Phi_1^\dagger + \Phi_2^\dagger$ into $\Phi_1^\dagger \in \ker T^*$ and $\Phi_2^\dagger = (\ker T^*)^\perp = \overline{\ran T}$.
	Then, we have
	\[ \norm{T^*\Phi^\dagger - \phi}_{\V'} - \scalprod{T^*\Phi^\dagger, \udag}_\X = \norm{T^*\Phi_2^\dagger - \phi}_{\V'} - \scalprod{T^*\Phi_2^\dagger, \udag}_\X \]
	and
	\[ \norm{\Phi^\dagger}_\Y^2 = \norm{\Phi_1^\dagger}_\Y^2 + \norm{\Phi_2^\dagger}_\Y^2. \]
	Now, if there was a component $\Phi_1^\dagger \neq 0$, it would follow that
	\begin{equation*}
		J_{T\udag}^\Y(\Phi_2^\dagger) = \frac{\norm{T^*\Phi_2^\dagger - \phi}_{\V'} - \scalprod{\Phi_2^\dagger, T\udag}_\Y}{\norm{\Phi_2^\dagger}_\Y}
		< \frac{\norm{T^*\Phi^\dagger - \phi}_{\V'} - \scalprod{\Phi^\dagger, T\udag}_\Y}{\norm{\Phi^\dagger}_\Y} = J_{T\udag}^\Y(\Phi^\dagger)
	\end{equation*}
	since $J_{T\udag}^\Y(\Phi^\dagger) < 0$ by Theorem 3.2 in \cite{KWW:2023},
	which contradicts the optimality of $\Phi^\dagger$.
\end{proof}

\begin{theorem}
	\label{inf_prior_inf_reg}
	Suppose that \cref{ass:V} holds, that $T$ is a compact, linear operator with singular system $(\tau_k,e_k,f_k)_{k\in\N}$, and that for all $k \in \N$,
	\[ \scalprod{T^*\Phi^\dagger,e_k}_\X = 0 \quad \Rightarrow \quad \scalprod{\phi,e_k}_\X = 0. \]
	Then we have
	\[ \inf_{C_0 \in F} J_{T\udag}^\Y \left(\Phimap(C_0)\right) = \min_{\Phi \in \Y \setminus\{0\}} J_{T\udag}^\Y(\Phi) < 0, \]
	where $F$ denotes the set of all operators $C_0$ on $\X$ of the form
	\[ C_0x = \sum_{k\in\N} \rho_k \scalprod{x,e_k}_\X e_k + C_1 P_{\ker T} \]
	with $(\rho_k)_{k \in \N} \in \ell^1(\N)$, $\rho_k > 0$ for all $k \in \N$ and a self-adjoint, positive definite, trace class operator $C_1$ on $\ker T$.
\end{theorem}
\begin{proof}
For $C_0 \in F$, we have
\[ \Phimap(C_0) = \sum_{k\in\N} \frac{\tau_k}{\tau_k^2 + \sigma^2\rho_k^{-1}} \scalprod{\phi,e_k}_\X f_k \]
by \eqref{eq:spectral_Phimap}.
By \cref{min_in_cl_ran_T}, $\Phi^\dagger \in \overline{\ran T} = \overline{\spn \{f_k: k \in \N\}}$. Therefore, $\Phimap = \Phi^\dagger$ if and only if for all $k \in \N$,
\begin{equation}
	\label{eq:cond_C0}
	\scalprod{\Phi^\dagger,f_k}_\Y = \frac{\tau_k}{\tau_k^2 + \sigma^2\rho_k^{-1}} \scalprod{\phi,e_k}_\X.
\end{equation}
This is trivially satisfied for $k \in \N$ with $\tau_k\scalprod{\Phi^\dagger,f_k} = \scalprod{T^*\Phi^\dagger,e_k} = 0$ by assumption, and equivalent to
\begin{equation}
	\label{eq:opt_rho_k}
	\rho_k^{-1} = \frac{\tau_k^2\scalprod{\phi - T^*\Phi^\dagger, e_k}_\X}{\sigma^2\scalprod{T^*\Phi^\dagger, e_k}_\X}
\end{equation}
for $k \in \N$ with $\scalprod{T^*\Phi^\dagger,e_k} \neq 0$.
We choose the sequence of operators
\[ C_{0,n} := \sum_{k\in\N} \rho_{k,n} \scalprod{\cdot,e_k}_\X e_k + C_1 P_{\ker T} \] 
on $\X$ with
\[ \rho_{k,n} := \begin{cases}
	\frac{\sigma^2 \scalprod{T^*\Phi^\dagger, e_k}}{\tau_k^2\scalprod{\phi - T^*\Phi^\dagger, e_k}} & \text{if}~k \le n~\text{and}~\scalprod{\phi - T^*\Phi^\dagger, e_k} \neq 0~\text{and}~\scalprod{T^*\Phi^\dagger,e_k} \neq 0, \\
	\sigma^2 n \tau_k^{-2} & \text{if}~k \le n~\text{and}~\scalprod{\phi - T^*\Phi^\dagger, e_k} = 0, \\
	\sigma^2 2^{-k} & \text{if}~k > n~\text{or}~\scalprod{T^*\Phi^\dagger,e_k} = 0,
\end{cases} \]
and an arbitrary self-adjoint, positive definite, trace class operator $C_1$ on $\ker T$.
This way, $C_{0,n}$ is positive definite and trace class for all $n \in \N$.
For $k \le n$ with $\scalprod{\phi - T^*\Phi^\dagger, e_k} \neq 0$ and $\scalprod{T^*\Phi^\dagger,e_k} \neq 0$, we have 
\[ \scalprod{\Phimap(C_{0,n}),f_k}_\Y = \scalprod{\Phi^\dagger,f_k}_\Y. \]
For $k \in \N$ with $\scalprod{T^*\Phi^\dagger,e_k} = 0$ we have
\[ \scalprod{\Phimap(C_{0,n}),f_k}_\Y = 0 = \scalprod{\Phi^\dagger,f_k}_\Y \]
by assumption.
For $k \le n$ with $\scalprod{\phi - T^*\Phi^\dagger, e_k} = 0$, we moreover have
\begin{align*} 
	\scalprod{\Phimap(C_{0,n}), f_k}_\Y &= \frac{\tau^k}{\tau_k^2 + \sigma^2\rho_k^{-1}} \left(\scalprod{\phi - T^*\Phi^\dagger, e_k}_\X + \scalprod{T^*\Phi^\dagger,e_k}_\X\right) \\
	&= \frac{1}{1 + \frac{1}{n}} \scalprod{\Phi^\dagger, f_k}_\Y
	= \frac{n}{n+1} \scalprod{\Phi^\dagger, f_k}_\Y.
\end{align*}
Now, it follows that
\begin{align*}
	&\bignorm{\Phimap(C_{0,n}) - \Phi^\dagger}_\Y^2 
	= \sum_{k=1}^\infty \scalprod{\Phimap(C_{0,n}) - \Phi^\dagger, f_k}_\Y^2 \\
	&\quad \le \sum_{k=1}^n \left(1 - \frac{n}{n+1}\right)^2 \scalprod{\Phi^\dagger, f_k}_\Y^2 + \sum_{k=n+1}^\infty \scalprod{\Phimap(C_{0,n}) - \Phi^\dagger, f_k}_\Y^2 \\
	&\quad = \frac{1}{(n+1)^2} \sum_{k=1}^n \scalprod{\Phi^\dagger, f_k}_\Y^2 + \sum_{k=n+1}^\infty \scalprod{\Phimap(C_{0,0}) - \Phi^\dagger, f_k}_\Y^2 \to 0
\end{align*}
as $n \to \infty$, since $\Phimap(C_{0,0}) - \Phi^\dagger \in \Y$.
By \cref{ass:V}.\ref{ass:Ts_bounded}, $J_{T\udag}^\Y$ is continuous on $\Y \setminus \{0\}$, so that
\[ J_{T\udag}^\Y \left(\Phimap(C_{0,n})\right) \to J_{T\udag}^\Y (\Phi^\dagger) = \min_{\Phi \in \Y} J_{T\udag}^\Y(\Phi) \]
as $n \to \infty$. This proves the proposition.
\end{proof}

\begin{remark}
	In general, the infimum can not be attained. If any component of $\phi$ satisfies 
	\[ \scalprod{\phi - T^*\Phi^\dagger, e_k}_\X = 0, \]
	then \eqref{eq:cond_C0} cannot hold because $\tau_k > 0$ for all $k \in \N$, and hence $\Phi^\dagger = \Phimap$ cannot hold either. Likewise, if the series
	\[ \sum_{k=1}^\infty \frac{\scalprod{T^*\Phi^\dagger, e_k}_\Y}{\tau_k^2\scalprod{\phi - T^*\Phi^\dagger, e_k}_\X} \]
	does not converge, then the optimal $C_0 \in F$ cannot be trace class because $(\rho_k)_{k\in\N}$, chosen according to \eqref{eq:opt_rho_k}, is not an $\ell^1$-sequence.
\end{remark}
\rev{In conjunction with equation \eqref{eq:size_map_test_alpha}, \cref{inf_prior_inf_reg} shows that MAP tests based upon a Gaussian prior distribution are, in principle, able to achieve the same power as the optimal regularized test.}


\section{Performance of MAP test under spectral source condition}
\label{sec:map_ssc}

Now, we study the performance of the MAP test $\Psimap$ in a more specific setting. Namely, under a spectral source condition and for a specific choice of the prior covariance $C_0$. We establish lower bounds on the power of the MAP test for an a priori choice of the prior covariance and we study its performance numerically for an a posteriori choice of the prior covariance.
We make the following assumptions.
\begin{assumption}
	\label{ass:spectral}
	\begin{enumerate}
		\item \label{ass:THS} $T$ is an injective Hilbert--Schmidt operator.
		\item \label{ass:source_cond} The truth $\udag$ satisfies a spectral source condition
		\[ \udag = \left(T^*T\right)^\frac{\nu}{2}w = \abs{T}^\nu w \]
		for some $\nu > 0$ and $w \in \X$ with $\norm{w}_\X \le \rho$.
		\item \label{ass:prior_cov} The prior covariance is given by
		\[ C_0 := \gamma^2(T^*T)^\mu \]
		for some $\gamma > 0$ and $\mu \ge 1$.
		\item The prior mean satisfies \eqref{cond_m0} for some $\alpha \in (0,1)$.
	\end{enumerate}
\end{assumption}
Under these assumptions, \cref{ass:V} is satisfied for $\V = \ran (T^*T)^{\nu/2}$, $\V' = \X$,
\[ \norm{v}_\V = \frac{1}{\rho}\bignorm{\left(T^*T\right)^{-\frac{\nu}{2}}v}_\X, \quad
	\norm{u}_{\V'} = \rho\bignorm{\left(T^*T\right)^{\frac{\nu}{2}}u}_\X. \]
As discussed in Section \ref{sec:reg}, this ensures that the MAP test is indeed a regularized test, i.e. $\Psimap = \Psi_{\Phi_\text{MAP},c_\text{MAP}}$, and that it has at most level $\alpha$.
	
Moreover, note that $C_0$ commutes with $T^*T$ under \Cref{ass:spectral}. 
\Cref{ass:spectral}.\ref{ass:THS} together with the condition $\mu \ge 1$ ensure that $(T^*T)^\mu$ is positive definite and has finite trace, which in turn ensures that the prior distribution $\Pi = \Normal(m_0,C_0)$ is concentrated on $\X$.
In order for $\V$ to have positive probability under the prior distribution, $\mu$ needs to be chosen large enough.

\begin{theorem}
	\label{trace_condition}
	If $(T^*T)^{\mu - \nu}$ is trace class and $m_0 \in \ran(T^*T)^{\nu/2}$, then $\Pi(\ran(T^*T)^{\nu/2}) = 1$. Otherwise, $\Pi(\ran(T^*T)^{\nu/2}) = 0$. 
\end{theorem}

The proof of this theorem can be found in \cref{sec:proofs_map_ssc}.
If $\mu \le \nu$, we have 
\[ \Pi(\ran (T^*T)^{\nu/2}) \le \Pi(\ran (T^*T)^{\mu/2}) = \Pi(\ran C_0^{1/2}) = 0 \]
and the prior satisfies the source condition with probability $0$ since the Cameron--Martin space of a Gaussian measure has probability $0$ by Proposition 1.27 in \cite{DaPrato:2006}.
On the other hand, we can guarantee that $\Pi(\ran (T^*T)^{\nu/2}) = 1$ by choosing $\mu$ large enough such that $(T^*T)^{\mu - \nu}$ is trace class and $m_0 \in \V$. Since $T$ is a Hilbert--Schmidt operator, this is the case if $\mu \ge \nu + 1$. Then, $\Pi$ can be interpreted as a Gaussian measure $\Normal(m_0,\gamma^2(T^*T)^{\mu - \nu})$ on $\V$. Note that $m_0$ can always be chosen in $\V$ as such that \eqref{cond_m0} is satisfied, e.g., as $m_0 = w(T^*T)^{\nu/2}\phi$ with $w \in \R$.
We will study both the compatible case $\mu \ge \nu + 1$ and the incompatible case $\mu \le \nu$.

\subsection{A priori choice of prior covariance}

\rev{By equation \cref{eq:size_map_test_alpha}, the power of the MAP test $\Psimap$ is given by
\begin{equation*}
	\P_\udag\left[\Psimap(Y) = 1\right] = Q\left(Q^{-1}(\alpha) - \frac{J_{T\udag}^\Y(\Phimap)}{\sigma}\right)
\end{equation*}
with $J_{T\udag}^Y$ as in \eqref{eq:def_J}.
For a prior covariance of the form $C_0 = \gamma^2(T^*T)^\mu$, the probe element $\Phimap$ is given by
\begin{equation*}
	\Phimap = \Phimap(\gamma) = T(T^*T)^\mu \left((T^*T)^{\mu + 1} + \frac{\sigma^2}{\gamma^2}\Id\right)^{-1}\phi.
\end{equation*}
For a given $\mu \ge 1$, one can now maximize the power of $\Psimap$ by minimizing 
\[ \gamma \mapsto J_{T\udag}^\Y(\Phimap(\gamma)). \]
However, this requires knowledge of $\udag$, so it results in an \emph{oracle test}.

Note that with the ansatz $\gamma = \gamma_0\sigma$, the probe element
\[ \Phimap(\gamma_0\sigma) = T(T^*T)^\mu \left((T^*T)^{\mu + 1} + \frac{1}{\gamma_0^{2}}\Id\right)^{-1}\phi \]
does no longer depend on the noise level. As a consequence, a probe element $\Phimap(\gamma_0\sigma)$ which is optimal for all noise levels can be found by minimizing
$\gamma_0 \mapsto J_{T\udag}^\Y(\Phimap(\gamma_0\sigma))$.
Thus, maximizing the power leads to a choice of $\gamma = \gamma_0\sigma$ in the order of the noise level $\sigma$.}

The power of $\Psimap$ can be expressed more explicitly as
\begin{align*}
	\P_\udag\left[\Psimap(Y) = 1\right] &= Q\left(Q^{-1}(\alpha) + \frac{\scalprod{\udag,T^*\Phimap}_\X - \norm{T^*\Phimap - \phi}_{\V'}}{\sigma\norm{\Phimap}_\Y}\right) \\
	&= Q\left(Q^{-1}(\alpha) + \frac{\scalprod{\udag,T^*\Phimap}_\X - \rho\bignorm{(T^*T)^\frac{\nu}{2}\left(T^*\Phimap - \phi\right)}_{\X}}{\sigma\norm{\Phimap}_\Y}\right).
\end{align*}
We can estimate the quantities in this expression in terms of the prior covariance as follows.

\begin{theorem}
	\label{residual_Phi_estimate}
	Suppose that \cref{ass:spectral} holds.
	If $\mu > \frac{\nu}{2} - 1$, then
	\begin{align*}
		\bignorm{\left(T^*T\right)^\frac{\nu}{2} \left(T^*\Phimap - \phi\right)}_\X &\le \left(\frac{\sigma}{\gamma}\right)^{\frac{\nu}{\mu + 1}} \norm{\phi}_\X, \\
		\bignorm{\Phimap}_\Y &\le \left(\frac{\gamma}{\sigma}\right)^{\frac{1}{\mu + 1}} \norm{\phi}_\X.
	\end{align*}
\end{theorem}

This theorem is proven in \cref{sec:proofs_map_ssc}.
It yields a lower bound on the power of $\Psimap$.
\begin{theorem}
	\label{power_a_priori}
	Suppose that \cref{ass:spectral} holds.
	If $\mu > \frac{\nu}{2} - 1$, then the power of $\Psimap$ for $\udag \in \X_K$ is bounded by
	\begin{equation}
		\label{eq:power_a_priori}
		\P_\udag\left[\Psimap(Y) = 1\right] \ge Q\left(Q^{-1}(\alpha) + \frac{\frac{\dualpair{\udag,\phi}}{\norm{\phi}} - 2\rho\gamma^{-\frac{\nu}{\mu + 1}}\sigma^{\frac{\nu}{\mu + 1}}}{\gamma^{\frac{1}{\mu + 1}}\sigma^{\frac{\mu}{\mu + 1}}}\right).
	\end{equation}
\end{theorem}
\begin{proof}
We estimate
\begin{align*}
	\dualpair{\udag,T^*\Phimap}_\X
	&= \dualpair{\udag,\phi}_\X + \dualpair{\udag,T^*\Phimap - \phi}_\X  \\
	&\ge \dualpair{\udag,\phi}_\X - \norm{\udag}_\V\bignorm{T^*\Phimap - \phi}_{\V'} \\
	&\ge \dualpair{\udag,\phi}_\X - \rho\bignorm{\left(T^*T\right)^\frac{\nu}{2} \left(T^*\Phimap - \phi\right)}_\X
\end{align*}
using Assumptions \ref{ass:V}.\ref{ass:CS} and \ref{ass:V}.\ref{ass:udag_bounded}.
Now, it follows from \cref{size_map_test,residual_Phi_estimate} that
\begin{align*}
	\P_\udag\left[\Psimap(Y) = 1\right] &\ge Q\left(Q^{-1}(\alpha) + \frac{\dualpair{\udag,\phi}_X - 2\rho\norm{\left(T^*T\right)^\frac{\nu}{2} \left(T^*\Phimap - \phi\right)}_\X}{\sigma\bignorm{\Phimap}_\Y}\right) \\
	&\ge Q\left(Q^{-1}(\alpha) + \frac{\dualpair{\udag,\phi}_\X - 2\rho(\sigma/\gamma)^{\frac{\nu}{\mu + 1}}\norm{\phi}_\X}{\sigma(\gamma/\sigma)^{\frac{1}{\mu + 1}}\norm{\phi}_\X}\right). \qedhere
\end{align*}
\end{proof}

For a specific noise level $\sigma > 0$ and a given $\gamma > 0$, both estimates in \cref{residual_Phi_estimate} are minimized by choosing $\mu$ as small as possible.
Consequently, the lower bound in \cref{power_a_priori} is maximized for any given $\gamma, \sigma > 0$ by choosing $\mu$ as small as possible subject to the imposed restrictions.
For $1 \le \nu < 4$, the optimal choice in the incompatible case is $\mu = 1 > \frac{\nu}{2} - 1$. The optimal choice for all $\nu > 0$ in the compatible case is $\mu = \nu + 1 > \frac{\nu}{2} - 1$.

If the normalized feature size $\dualpair{\udag,\phi}/\norm{\phi}$ exceeds $2\rho(\sigma/\gamma)^{\nu/(\mu + 1)}$, \cref{power_a_priori} moreover guarantees that the MAP test has nontrivial power, i.e., that $\P_\udag\left[\Psimap(Y) = 1\right] > \alpha$.
In particular, $\Phimap$ has nontrivial power below a certain noise level.
For a fixed noise level, $\gamma$ can be chosen such that this detection threshold is below a certian value.

\begin{corollary}
	Suppose that \cref{ass:spectral} holds and that $\mu > \frac{\nu}{2} - 1$.
	For any a priori choice of the form
	\[ \gamma = \gamma_0 \sigma^\omega \]
	with $\gamma_0 > 0$ and $\omega \in (-\mu,1)$, the power of $\Phimap$ for $\udag \in \X_K$ converges to $1$ as $\sigma \to 0$.
\end{corollary}
\begin{proof}
	For any such choice of $\gamma$, both the term
	\[ 2\rho\gamma^{-\frac{\nu}{\mu + 1}}\sigma^{\frac{\nu}{\mu + 1}} = 2\rho\gamma_0^{-\frac{\nu}{\mu + 1}}\sigma^{\frac{(1 - \omega)\nu}{\mu + 1}} \]
	and the term
	\[ \gamma^\frac{1}{\mu + 1}\sigma^\frac{\mu}{\mu + 1} = \gamma_0^\frac{1}{\mu +1}\sigma^\frac{\mu + \omega}{\mu + 1} \]
	in \eqref{power_a_priori} converge to $0$ as $\sigma \to 0$. Hence, the fraction
	\[ \frac{\frac{\dualpair{\udag,\phi}}{\norm{\phi}} - 2\rho\gamma^{-\frac{\nu}{\mu + 1}}\sigma^{\frac{\nu}{\mu + 1}}}{\gamma^{\frac{1}{\mu + 1}}\sigma^{\frac{\mu}{\mu + 1}}} \]
	becomes positive and converges to $\infty$ for $\udag \in \X_K$ as $\sigma$ tends to $0$. Now, it follows that 
	\[ \lim_{\sigma \to 0} \P_\udag\left[\Psimap(Y) = 1\right] = 1 \]
	from \cref{power_a_priori} and the fact that $Q(t) \to 1$ as $t \to \infty$.
\end{proof}

Note that \cref{power_a_priori} does not guarantee convergence of the power to $1$ for $\omega = 1$, since the term $2\rho(\gamma/\sigma)^{1/(\mu + 1)} = 2\rho\gamma_0^{1/(\mu + 1)}$ does not converge to $0$ in this case.
\Cref{power_a_priori} yields the following bound to the separation rate of $\Phimap$.

\begin{corollary}
	Suppose that \cref{ass:spectral} holds and set
	\rev{\[ \xi_\sigma := \sigma^\frac{\nu}{\nu + 1} \quad \text{for all}~\sigma > 0. \]}
	Let $(u_\sigma^\dagger)_{\sigma \in (0,1]}$ be a family in $\X_K$ such that
	\[ \lim_{\sigma \to 0} \dualpair{u_\sigma^\dagger,\phi}_\X = 0 \quad \text{and} \quad \lim_{\sigma \to 0} \frac{\dualpair{u_\sigma^\dagger,\phi}_\X}{\xi_\sigma} \to \infty. \]
	If $\mu > \frac{\nu}{2} - 1$ and $\gamma$ is chosen a priori as
	\[ \gamma = \sigma^\frac{\nu - \mu}{\nu + 1}, \]	
	then the power of $\Phimap$ for $u_\sigma^\dagger$ converges to $1$ as $\sigma \to 0$.
\end{corollary}
\begin{proof}
	We estimate the power of $\Phimap$ using \cref{power_a_priori}.
	First, note that $\omega := \frac{\nu - \mu}{\nu + 1}$ satisfies
	\[ 1 - \omega = \frac{\mu + 1}{\nu + 1}. \]
	This allows us to express the second term in the numerator of the fraction on the right hand side of \eqref{eq:power_a_priori} as
	\rev{\[ 2\rho\gamma^{-\frac{\nu}{\mu + 1}}\sigma^\frac{\nu}{\mu + 1} 
	= 2\rho\sigma^{\frac{(1 - \omega)\nu}{\mu + 1}}
	= 2\rho\sigma^\frac{\nu}{\nu + 1}
	= 2\rho\xi_\sigma. \]}
	Moreover, we have
	\[ \mu + \omega = \frac{(\mu + 1)\nu}{\nu + 1}, \]
	which lets us express the denominator of said fraction as
	\rev{\[ \gamma^\frac{1}{\mu + 1}\sigma^\frac{\mu}{\mu + 1} 
	= \sigma^\frac{\omega + \mu}{\mu + 1}
	= \sigma^\frac{\nu}{\nu + 1} = \xi_\sigma. \]
	Now, the conditions on $u_\sigma^\dagger$ yield
	\[ \lim_{\sigma \to 0} \frac{\frac{\dualpair{u_\sigma^\dagger,\phi}}{\norm{\phi}} - 2\rho\gamma^{-\frac{\nu}{\mu + 1}}\sigma^{\frac{\nu}{\mu + 1}}}{\gamma^{\frac{1}{\mu + 1}}\sigma^{\frac{\mu}{\mu + 1}}}
	= \lim_{\sigma \to 0} \frac{\dualpair{u_\sigma^\dagger,\phi}}{\norm{\phi}\xi_\sigma} - 2\rho = \infty. \]}
	This implies that the right hand side of \eqref{eq:power_a_priori} converges to $1$ as $\sigma$ tends to $0$.
\end{proof}

\begin{remark}
The choice of $m_0$ according to \eqref{cond_m0} has to take into account the worst case
\[ \scalprod{(T^*T)^\frac{\nu}{2}(T^*\Phimap - \phi), (T^*T)^{-\frac{\nu}{2}}u^\dagger} = \norm{(T^*T)^\frac{\nu}{2}(T^*\Phimap - \phi)} \cdot \norm{(T^*T)^{-\frac{\nu}{2}}u^\dagger} \]
to guarantee level $\alpha$, i.e., when
\[ \udag = (T^*T)^\nu(T^*\Phimap - \phi). \]
When estimating the power, in contrast, the worst case
\[ -\scalprod{(T^*T)^\frac{\nu}{2}(T^*\Phimap - \phi), (T^*T)^{-\frac{\nu}{2}}u^\dagger} = \norm{(T^*T)^\frac{\nu}{2}(T^*\Phimap - \phi)} \cdot \norm{(T^*T)^{-\frac{\nu}{2}}u^\dagger} \]
has to be taken into account, i.e., when $(T^*T)^{-\frac{\nu}{2}}u^\dagger = (T^*T)^\frac{\nu}{2}(\phi - T^*\Phi)$.
Without relating the two terms $\scalprod{\phi,u^\dagger}$ and $\scalprod{(T^*T)^\frac{\nu}{2}(T^*\Phimap - \phi), (T^*T)^{-\frac{\nu}{2}}u^\dagger}$ and without further knowledge about $u^\dagger$ other than the source condition, no better estimate for the power is possible, and the loss in power corresponding to the term $-2\rho\sigma^{-1}\norm{(T^*T)^{\nu/2}(T^*\Phimap - \phi)}/\norm{\Phimap}$ is unavoidable.
\end{remark}

For a fixed noise level, $\gamma$ can be chosen to maximize the lower bound to the power in \cref{power_a_priori}. The optimal value of $\gamma$, however, depends on the normalized feature size $\dualpair{\udag,\phi}/\norm{\phi}$, which is unknown.
Nonetheless, we can choose $\gamma$ optimally and maximize the right hand side of \eqref{eq:power_a_priori} for a certain \emph{expected} normalized feature size $\xi > 0$. This leads to the following result.

\begin{theorem}
	\label{bound_power_xi}
	Suppose that \cref{ass:spectral} holds.
	If $\gamma$ is chosen a priori as
	\begin{equation}
		\label{eq:a_priori_choice_gamma}
		\gamma = \left(\frac{\xi}{2\rho(\nu + 1)}\right)^{-\frac{\mu + 1}{\nu}}\sigma
	\end{equation}
	with $\xi > 0$ and if $\mu > \frac{\nu}{2} - 1$, then the power of $\Psimap$ for $\udag \in \X_K$ is bounded by
	\begin{equation}
		\label{eq:power_xi}
		\P_\udag\left[\Psimap(Y) = 1\right] \ge Q\left(Q^{-1}(\alpha) + \frac{\left(\frac{\dualpair{\udag,\phi}}{\norm{\phi}} - \frac{\xi}{\nu + 1}\right) \left(\frac{\xi}{\nu + 1}\right)^\frac{1}{\nu}}{\sigma(2\rho)^\frac{1}{\nu}}\right).
	\end{equation}	
\end{theorem}
\begin{proof}
We maximize the lower bound in \cref{power_a_priori} over $\gamma$ by maximizing
\[ f(\gamma) := \left(\xi - 2\rho\sigma^{\frac{\nu}{\mu + 1}}\gamma^{-\frac{\nu}{\mu + 1}}\right)\gamma^{-\frac{1}{\mu + 1}}. \]
The derivative
\begin{align*}
	f'(\gamma) &= 2\rho\sigma^{\frac{\nu}{\mu + 1}}\frac{\nu}{\mu + 1}\gamma^{-\frac{\nu + \mu + 1}{\mu + 1}} \gamma^{-\frac{1}{\mu + 1}}
	- \left(\xi - 2\rho\sigma^{\frac{\nu}{\mu + 1}}\gamma^{-\frac{\nu}{\mu + 1}}\right) \frac{1}{\mu + 1}\gamma^{-\frac{\mu + 2}{\mu + 1}} \\
	&= \left(2\rho\sigma^{\frac{\nu}{\mu + 1}}(\nu + 1)\gamma^{-\frac{\nu}{\mu + 1}} - \xi\right) \frac{1}{\mu + 1}\gamma^{-\frac{\mu + 2}{\mu + 1}}
\end{align*}
is equal to zero if and only if
\begin{equation*}
	\label{eq:choice_gamma0}
	\gamma = \left(\frac{\xi}{2\rho(\nu + 1)}\right)^{-\frac{\mu + 1}{\nu}}\sigma =: \bar{\gamma}(\xi,\rho,\nu,\mu,\sigma).
\end{equation*}
$f'$ is positive on $(0,\bar{\gamma})$ and negative on $(\bar{\gamma},\infty)$. Therefore, $f$ attains its maximum
\begin{equation*}
	f(\bar{\gamma}) = \left(\frac{\dualpair{\udag,\phi}}{\norm{\phi}} - 2\rho\sigma^{\frac{\nu}{\mu + 1}}\bar{\gamma}^{-\frac{\nu}{\mu + 1}}\right)\bar{\gamma}^{-\frac{1}{\mu + 1}} 
	= \left(\frac{\dualpair{\udag,\phi}}{\norm{\phi}} - \frac{\xi}{\nu + 1}\right) \left(\frac{\xi}{2\rho(\nu + 1)}\right)^\frac{1}{\nu} \sigma^{-\frac{1}{\mu + 1}}
\end{equation*}
in $\bar{\gamma}$.
This leads to the lower bound
\begin{equation*}
	\beta(\udag) \ge Q\left(Q^{-1}(\alpha) + \frac{\left(\frac{\dualpair{\udag,\phi}}{\norm{\phi}} - \frac{\xi}{\nu + 1}\right) \left(\frac{\xi}{\nu + 1}\right)^\frac{1}{\nu}}{\sigma(2\rho)^\frac{1}{\nu}}\right)
\end{equation*}
for the power of $\Psimap$ for every $\xi > 0$ and any $\udag \in \X_K$.
\end{proof}

The lower bound \eqref{eq:power_xi} guarantees a nontrivial power for any $\udag \in \X_K$ whose normalized feature size $\dualpair{\udag,\phi}_\X/\norm{\phi}_\X$ exceeds $\xi/(\nu + 1)$ and maximizes \eqref{eq:power_a_priori} if the normalized feature size is equal to $\xi$. 
Note that the optimal choice of $\gamma$ for a given $\xi$ is proportional to $\sigma$.
Note, moreover, that the resulting lower bound is independent of the choice of $\mu$.
On the one hand, \cref{bound_power_xi} provides an oracle lower bound for the power of $\Phimap$ with the choie $\xi = \dualpair{\udag,\phi}_\X/\norm{\phi}_\X$.
On the other hand, $\xi$ can be used to calibrate the detection threshold of the test.
Due to the source condition, we know that
\[ \frac{\dualpair{\udag,\phi}_\X}{\norm{\phi}_\X} \le \bignorm{(T^*T)^\frac{\nu}{2}w}_\X \le \rho\norm{T^*T}^\frac{\nu}{2}, \]
so that only a choice $\xi \in (0,\rho\norm{T^*T}^{\nu/2}]$ is reasonable.

\subsection{A posteriori choice of prior covariance}

\rev{In the previous section, we have seen that one can define an optimal oracle MAP test for a given $\mu \ge 1$ by minimizing 
\[ \gamma_0 \mapsto J_{T\udag}^\Y(\Phimap(\gamma_0\sigma)). \]}
However, the functional $J_{T\udag}^\Y$ is unaccessible in practice because its evaluation requires knowledge of $\udag$. For this reason, we apply a similar approach as in \cite{KWW:2023}.
We estimate $\scalprod{T\udag,\Phi}_\Y$ by $\dualpair{Y,\Phi}$ and replace $J_{T\udag}^\Y$ by the empirical functional
\[ J_Y^\Y(\Phi) = \frac{\norm{T^*\Phi - \phi}_{\V'} - \dualpair{Y, \Phi}_{\Y^* \times \Y}}{\norm{\Phi}_\Y} \]
which does not require knowledge of $\udag$.
\rev{Then, we choose $\gamma_0 > 0$ as minimizer of the penalized objective functional
\[ \gamma_0 \mapsto J_Y^\Y\left(\Phimap(\gamma_0\sigma)\right) + \omega\left(\log \gamma_0\right)^2 \]
with regularization parameter $\omega > 0$ and set $\gamma = \gamma_0\sigma$.}
Here, the main purpose the penalty is to prevent a choice of $\Phimap$ with huge norm in case that $J_Y^\Y$ is a positive functional.

Since this choice of $\gamma$ depends on $Y$, the power of the resulting test $\Psimap$ is no longer given by \eqref{eq:size_map_test_alpha} when it is applied to the same data $Y$. Instead, \eqref{eq:size_map_test_alpha} describes the power of the MAP test constructed using one realization $Y_1$ of the data and applied to another independent realization $Y_2$.
In the following, we will distinguish between the $1$ sample MAP test, which is applied to the same data used for its construction, and the $2$ sample MAP test, which is applied to another independent data sample.


\section{Numerical simulations}

In this section, we will perform numerical simulations investigating the frequentistic power of the tests derived in the previous sections over a range of noise levels. For this purpose, we consider three different test problems. 

\subsection{Test problems}

\subsubsection{Deconvolution}
\label{sec:deconv}

The convolution between two functions $h \in L^1(\R)$ and $u \in L^2(\R)$ is defined by
\begin{equation*}
	(h * u)(x) = \int_{\R} h(x - z) u(z) \di z \quad \text{for all}~x \in \R.
\end{equation*}
The Fourier transform of $h \in L^1(\R)$ is, moreover, defined by
\begin{equation*}
	(\Fourier h)(\xi) = \int_{\R} h(x) \exp \left(2\pi i\scalprod{x,\xi}{}\right) \di x \quad \text{for all}~\xi \in \R.
\end{equation*}
We initially consider the convolution operator 
\[
T_{\mathrm{conv}} : L^2(\R) \to L^2(\R), \qquad Tu := h * u
\]
with the kernel $h \in L^1(\R)$ defined in terms of its Fourier transform
\begin{equation}
	\label{eq:def_h}
	(\Fourier h)(\xi) = \left(1 + 0.06^2\xi^2\right)^{-2} \quad \text{for all}~\xi \in \R.
\end{equation}

The convolution between $h$ and $u$ can be approximated by the \emph{periodic convolution}
\[ (\tilde{h} *_P \tilde{u})(x) = \int_{-\frac{P}{2}}^{\frac{P}{2}} \tilde{h}(x - z)\tilde{u}(z) \di z, \quad x \in \R, \]
between the $P$-\emph{periodization} $\tilde{h}$ of $h$,
\begin{equation*}
	\tilde{h}(x) := h_{\text{per},P}(x) := \sum_{l \in \Z} h(x + lP) \quad \text{for all}~x \in \R,
\end{equation*}
and the $P$-periodization $\tilde{u} := u_{\text{per},P}$ of $u$.
In the following, we periodize with $P := 2$ and consider the \emph{periodic} convolution operator $T_{\mathrm{conv},P}\tilde{u} := \tilde{h} *_P \tilde{u}$ on 
\[ \X := \Y := L^2(-1,1) \]
associated with the periodized kernel $\tilde{h} = h_{\text{per},P} \in L^1(-1,1)$. Throughout, we moreover assume that $\esssupp \udag \subseteq [-\frac12,\frac12]$.
By \cref{norm_conv_op}, $T_{\mathrm{conv},P}$ is bounded with
\[ \norm{T_{\mathrm{conv},P}}_{L^2 \to L^2} = 1. \]
We discretize the problem using a uniform grid of size $N = 1024$ and compute the discrete convolution using a fast Fourier transform. For more details on the implementation of the convolution see Appendix A in \cite{KWW:2023}.

\subsubsection{Differentiation}
\label{sec:differentiation}

As a second example, we consider the problem to estimate the second weak derivative of a function $y \in H^2(0,1)$. For $m \in \N$, let us define a linear operator $\widetilde{T}_m$: $L^2(-1,1) \to L^2(-1,1)$ by
\[ [\widetilde{T}_mf](x) := \hat{f}(0) + \sum_{k \in \Z \setminus \{0\}} -\left(\pi i k\right)^{-m} \hat{f}(k) \exp \left(\pi i k x\right) \quad \text{for all}~x \in (-1,1), \]
where $\hat{f}(k)$ denotes the $k$th Fourier coefficient
\begin{equation}
	\label{eq:fhat}
	\hat{f}(k) := \frac12 \int_{-1}^1 f(x) \exp \left(-\pi i k x\right) \di x \quad \text{for all}~k \in \Z.
\end{equation}
By differentiating this series term-wise, we find that
\[ [\widetilde{T}_mf]^{(m)}(x) = \sum_{k \in \Z \setminus \{0\}} -\hat{f}(k) \exp \left(\pi i k x\right)
	= - \sum_{k \in \Z} \hat{f}(k) \exp \left(\pi i k x\right) = -f(x) \]
almost everywhere for all $f \in L^2(-1,1)$ with $\hat{f}(0) = 0$.
In particular, $\ran \widetilde{T}_m \subseteq H^m(-1,1)$.
Let us denote the extension operator of a function $u \in L^2(0,1)$ to an odd function on $(-1,1)$ by $E$,
\[ [Eu](x) = \begin{cases}
	-u(-x) & \text{for all}~x \in (-1,0), \\
	0 & \text{for}~x = 0, \\
	u(x) & \text{for all}~x \in (0,1),
\end{cases} \]
and the restriction of a function on $(-1,1)$ to a function on $(0,1)$ by $R$.
Then, $[R\widetilde{T}_mEu]^{(m)} = -u$ for all $u \in L^2(0,1)$ since $\hat{f}(0) = 0$ for odd functions $f \in L^2(-1,1)$.
Now, the problem of estimating the negative $m$th weak derivative $-y^{(m)}$ of $y \in H^m(0,1)$ can be formulated as solving the equation $Tu = y$ in $L^2(0,1)$ with the antiderivative operator
\[
T_{\mathrm{antider},m} : L^2(0,1) \to L^2(0,1), \qquad T_{\mathrm{antider},m} := R\widetilde{T}_mE.
\]
This representation of the forward operator has the advantage that it can be implemented efficiently using a fast Fourier transform.

The condition that $[\widetilde{T}_mf]^{(m)} = -f$ almost everywhere on $(-1,1)$ uniquely determines the eigenvalues $\lambda_k = -(\pi i k)^{-m}$, $k \in \Z \setminus \{0\}$, of the operator $\widetilde{T}_m$ since its eigenvectors $\{\exp(\pi i k \cdot)\}_{k\in\Z \setminus \{0\}}$ form an orthonormal basis of $\{f \in L^2(-1,1): \hat{f}(0) = 0\}$. It does not determine the eigenvalue $\lambda_0$ of $\widetilde{T}$ associated with the space of constant functions, which is chosen arbitrarily as $1$. The definition of the operator $T$ as $R\widetilde{T}E$ is, however, independent of the choice of $\lambda_0$ since $\widehat{Eu}(0) = 0$ for all $u \in L^2(0,1)$.
As a consequence, $\lambda_1 = -(\pi i)^{-m}$ and $\lambda_{-1} = -(-\pi i)^{-m}$ are the largest eigenvalues of $T_{\mathrm{antider},m}$ (in absolute value), so that its norm is given by
\[ \norm{T_{\mathrm{antider},m}}_{L^2 \to L^2} = \abs{\lambda_1} = \pi^{-m}. \]

In the following, we restrict ourselves to the case $m = 2$. Here, $T_{\text{antider},2}$ can also be expressed as a Fredholm operator, see \cref{sec:fredholm}.
In order to make the differentiation problem comparable with the deconvolution problem, we use a normalized version of the operator $T_{\mathrm{antider},2}$ for our numerical simulations. Specifically, we scale the operator such that its norm is equal to $1$.
We discretize the problem using a uniform grid 
\begin{equation}
	\label{eq:grid}
	\left\{ \frac{1}{N+1}, \dots, \frac{N}{N+1} \right\}
\end{equation}
of size $N = 1024$.

\subsubsection{Backward heat equation} 

As a third example, we consider the backward heat equation problem on the interval $(0,1)$ with Dirichlet boundary conditions. Given measurements of $y = f \left(\cdot, t_0\right)$ at time $t_0 > 0$, we want to find the initial heat distribution $\udag$ in
\[
\begin{cases} \frac{\partial f}{\partial t} \left(x,t\right) = \frac{\partial^2 f}{\partial t^2} \left(x,t\right) & \mathrm{in} \left(0,1\right) \times \left(0,t_0\right), \\ 
	f \left(x,0\right) = u^\dagger(x) & \mathrm{on} \left[0,1\right], \\
	f \left(0,t\right) = f\left(1,t\right) = 0 & \mathrm{on} \left[0,t_0\right],\end{cases}
\]
see e.g. \cite{w18}. This gives rise to the solution operator 
\[ T_\mathrm{heat} : L^2(0,1) \to L^2(0,1), \qquad T_\mathrm{heat} u^\dagger := f \left(\cdot, t_0\right). \]

Numerically, we implement this operator similarly as the antiderivative operator using a fast Fourier transform. It can be written as
$T_\mathrm{heat} = R\widetilde{T}_\mathrm{heat}E$,
where $\widetilde{T}_\mathrm{heat}$: $L^2(-1,1) \to L^2(-1,1)$ is defined as
\[ [\widetilde{T}_\mathrm{heat}f](x) := \sum_{k \in \Z} \exp\left(-\pi^2 t_0 k^2\right) \hat{f}(k) \exp(\pi ikx) \quad \text{for all}~x \in (-1,1), \]
and $\hat{f}(k)$ denotes the $k$th Fourier coefficient of $f$ as in \eqref{eq:fhat}, see \cite{w18}.
Since $\widehat{Eu}(0) = 0$ for all $u \in L^2(0,1)$, the largest eigenvalues (in absolute value) of $T_\mathrm{heat}$ are $\lambda_1 = \exp(-\pi^2 t_0)$ and $\lambda_{-1} = \exp(-\pi^2 t_0)$. Consequently, its norm is given by
\[ \norm{T_\mathrm{heat}}_{L^2 \to L^2} = \abs{\lambda_1} = \exp(-\pi^2 t_0). \]
Note that compared to the deconvolution and the differentiation problem, the backward heat equation determines a severely ill-posed forward operator in the sense that its singular values decay exponentially.

Throughout, we set $t_0 := 10^{-4}$.
As in the differentiation problem, we normalize the operator $T_{\mathrm{heat}}$ for our numerical simulations to have norm $1$ so that the problem is comparable to the previous two.
We discretize the problem using the uniform grid \eqref{eq:grid} of size $N = 1024$.

\subsection{Considered scenarios}

As the feature of interest, we consider the $L^2$-inner product of a function $\udag$ with a symmetric beta kernel on the interval $[c,c+l]$. We choose
\[ \phi(x) := p_{c,l}(x;\beta) \propto \begin{cases}
	x^{\beta - 1}(l - x)^{\beta - 1} & \text{if}~x \in [c,c+l], \\
	0 & \text{else},
\end{cases} \]
where $\beta > 0$, $c \in \R$, and $l := \frac{5}{128}$. Throughout, $p_{c,l}$ is normalized such that $\norm{p_{c,l}}_{L^2(\R)} = 1$. Note that for $\beta = 1$, $p_{c,l}(\cdot,1)$ is simply the indicator function of the interval $[c,c+l]$. We choose $c$ as the center of the considered interval, i.e., as $c := 0$ for the deconvolution problem and $c := \frac12$ for the differentiation problem and the backward heat equation.

For the deconvolution problem, the choice $\phi = p_{c,l}$ can be interpreted as performing support inference, or more specifically, deciding if $\supp \udag \cap (c,c+l) = \emptyset$ (under nonnegativity constraints). For the differentiation problem, the choice $\phi = p_{c,l}$ corresponds to deciding whether $y^\dagger = T_{\mathrm{antider},2}\udag$ is linear on the interval $(c, c+l)$ (under concavity constraints on $y^\dagger$).

For all three problems, we consider a choice of $\beta = 1$, which results in an \emph{incompatible scenario} with $\phi \notin \ran T^*$.
For the deconvolution and the differentiation problem, we additionally consider a choice of $\beta = \beta_{\mathrm{conv}} := 5$ and $\beta = \beta_{\mathrm{antider}} := 3$, respectively, which guarantees a \emph{compatible scenario} with $\phi \in \ran T^*$, in which the unregularized test $\Psi_0$ is well-defined. As argued in \cite{KWW:2023}, the operator $T_{\mathrm{conv}}$ satisfies $\ran T_{\mathrm{conv}} \subseteq H^4(-1,1)$ for the periodic Sobolev space \[ H^t(-1,1) = \left\{f \in L^2(-1,1): k \mapsto (1 + k^2)^\frac{t}{2}\widehat{f}(k) \in \ell^2(\Z)\right\} \]
equipped with the norm
\[ \norm{f}_{H^t} = P^\frac12 \left(\sum_{k \in \Z} \left(1 + k^2\right)^t \bigabs{\widehat{f}(k)}^2\right)^\frac12. \]
Since for $\beta \in \N$, the function $\phi$ is $(\beta-1)$-times differentiable, it holds $\phi \in H^{\beta - 1}(-1,1)$ and furthermore if we choose $\beta = \beta_{\mathrm{conv}} = 5$, then we obtain $\phi \in \ran T_{\mathrm{conv}}^*$.  Similarly, it is clear by definition that $\ran T_{\mathrm{antider},2} \subseteq H^2(-1,1)$, and by choosing $\beta = \beta_{\mathrm{antider}} = 3$ we obtain $\phi \in \ran T_{\mathrm{antider},2}^*$.
Note, that due to the infinite smoothing properties of $T_{\mathrm{heat}}$, there is no $\beta \in\N$ such that $\phi \in \ran T_{\mathrm{heat}}^*$. I.e., for the backward heat equation there is no compatible scenario.
In the incompatible scenario, we compute a probe element $\Phi_0$ for the unregularized test numerically as the minimum norm solution to the equation $T^*\Phi = \phi$ despite its ill-posedness.

For the deconvolution and the differentiation problem, we choose a symmetric beta kernel on the interval $[c+\lambda,c+\lambda+l]$ as the truth,
\[ \udag(x) := p_{c+\lambda,l}(x;\delta) \quad \text{for all}~x \in \R, \]
where $\delta > 0$ and $\lambda := \frac13 l = \frac{5}{384}$.
In the devoncolution problem, we choose $\delta = \delta_\mathrm{conv} := 5$.
In the differentiation problem, we choose $\delta = \delta_\mathrm{antider} := 3$.
For the backward heat equation, we define a source element $w$ by
\[ w(x) \propto \begin{cases}
	1 & \text{if}~x \in [c + \lambda, c + \lambda + l), \\
	-1 & \text{if}~x \in [c + \lambda - \frac{l}{2}, c + \lambda + \frac{3l}{2}) \setminus [c + \lambda, c + \lambda + l), \\
	0 & \text{else},
\end{cases} \]
set $\udag = (T^*T)^{\frac12}w$ and normalize $w$ such that $\norm{\udag}_{L^2(0,1)} = 1$. This way, \cref{ass:spectral}.\ref{ass:source_cond} is satisfied with $\nu = 1$ in all three problems.

\setlength{\fwidthscen}{4.2cm}
\setlength{\fheightscen}{3.2cm}
\setlength{\fwidth}{6.5cm}
\setlength{\fheight}{3.2cm}
\newcommand\sigmamindeconv{1e-5}
\newcommand\sigmamindiff{1e-6}
\newcommand\sigmaminheat{1e-5}
\newcommand\sigmamaxdeconv{1}
\newcommand\sigmamaxdiff{0.1}
\newcommand\sigmamaxheat{1}
\definecolor{myblue}{rgb}{0.00000,0.44700,0.74100}
\definecolor{myred}{rgb}{0.85000,0.32500,0.09800}
\definecolor{myyellow}{rgb}{0.92900,0.69400,0.12500}
\definecolor{mygreen}{rgb}{0,0.66667,0}
\definecolor{mylightblue}{rgb}{0.5,0.7235,0.8705}

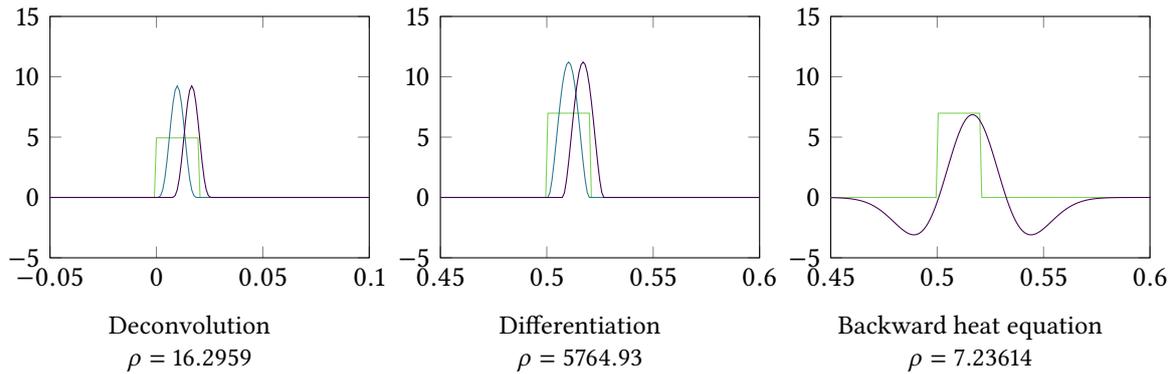
\begin{figure}[!htb]
	\centering
	\subcaptionbox*{Deconvolution\\$\rho = 16.2959$}[0.32\textwidth]{\begin{tikzpicture}[baseline]
	\begin{axis}[%
		width=\fwidthscen,
		height=\fheightscen,
		scale only axis,
		xmin=-0.05,
		xmax=0.1,
		ymin=-5,
		ymax=15,
		xticklabel style={/pgf/number format/.cd,fixed,precision=2},
	]
		\addplot[color of colormap=800 of viridis] table [x index=0, y index=1] {data/scenario_prob1_phi1.dat};
		\label{plot:phi_nonsmooth}
		\addplot[color of colormap=400 of viridis] table [x index=0, y index=1] {data/scenario_prob1_phi5.dat};
		\label{plot:phi_smooth}
		\addplot[color of colormap=0 of viridis] table [x index=0, y index=2] {data/scenario_prob1_phi1.dat};
		\label{plot:udag}
	\end{axis}
\end{tikzpicture}%
}
	\subcaptionbox*{Differentiation\\$\rho = 5764.93$}[0.32\textwidth]{\begin{tikzpicture}[baseline]
	\begin{axis}[%
		width=\fwidthscen,
		height=\fheightscen,
		scale only axis,
		xmin=0.45,
		xmax=0.6,
		ymin=-5,
		ymax=15,
		xticklabel style={/pgf/number format/.cd,fixed,precision=2},
	]
		\addplot[color of colormap=800 of viridis] table [x index=0, y index=1] {data/scenario_prob2_phi1.dat};
		\addplot[color of colormap=400 of viridis] table [x index=0, y index=1] {data/scenario_prob2_phi3.dat};
		\addplot[color of colormap=0 of viridis] table [x index=0, y index=2] {data/scenario_prob2_phi1.dat};
	\end{axis}
\end{tikzpicture}%
}
	\subcaptionbox*{Backward heat equation\\$\rho = 7.23614$}[0.32\textwidth]{\begin{tikzpicture}[baseline]
	\begin{axis}[%
		width=\fwidthscen,
		height=\fheightscen,
		scale only axis,
		xmin=0.45,
		xmax=0.6,
		ymin=-5,
		ymax=15,
		xticklabel style={/pgf/number format/.cd,fixed,precision=2},
	]
		\addplot[color of colormap=800 of viridis] table [x index=0, y index=1] {data/scenario_prob3_phi1.dat};
		\addplot[color of colormap=0 of viridis] table [x index=0, y index=2] {data/scenario_prob3_phi1.dat};
	\end{axis}
\end{tikzpicture}%
}
	\caption{The function $\phi_{l,\beta}$ with $l = \frac{5}{128}$ for $\beta = 1$ (\ref{plot:phi_nonsmooth}) and the corresponding $\beta_{\mathrm{conv}}, \beta_{\mathrm{antider}}$ chosen such that $\phi_{l,\beta} \in \ran T^*$ (\ref{plot:phi_smooth}), and the truth $u^\dagger$ (\ref{plot:udag}).}
	\label{fig:scenarios}
\end{figure}

\subsection{Computation of power and level}

The exact power of the unregularized test is computed using \eqref{eq:size_reg_test_alpha}.
The exact power of the oracle MAP tests with optimal choice of $\gamma$ and with a priori choice of $\gamma$ according to \eqref{eq:a_priori_choice_gamma} and $\xi = \scalprod{\udag,\phi}/\norm{\phi}$ are computed using \eqref{eq:size_map_test_alpha}. Moreover, the lower bound for the power of the oracle MAP test from \cref{bound_power_xi} is computed for comparison using \eqref{eq:power_xi}.
For all these tests, the rejection threshold $c$ is chosen according to \eqref{eq:c_reg} to guarantee a level of at most $\alpha = 0.1$ under \cref{ass:spectral}.

For the a posteriori choice of $\gamma$, the penalty term is weighted with $\omega = 10^{-4}$ in case of the deconvolution and heat equation problem, and with $\omega = 10^{-10}$ in case of the differentiation problem. In all cases, $\omega$ was chosen roughly as large as possible while ensuring covergence of the power of the MAP test to $1$ in our simulations.

The power of the MAP test with a posteriori choice of $\gamma$ is computed empirically using $M_\text{power} = 1000$ data samples as rejection percentage. 
Its rejection threshold $c$ is chosen according to \eqref{eq:c_reg} with a conservative choice of $\alpha = \alpha_1 := 0.05$.
In addition, the power of the MAP test with a posteriori choice of $\gamma$ applied to a second data sample, independent from that used for its construction, is computed empirically as
\[ \frac{1}{M_\text{power}} \sum_{m = 1}^{M_\text{power}} Q\left(Q^{-1}(\alpha) - \frac{J_{T\udag}^\Y(\Phimap(y_m))}{\sigma}\right) \]
using $M_\text{power} = 1000$ data samples $y_1, \dots, y_m$, see also section 6.1 in \cite{KWW:2023}.
Here, the rejection threshold is again chosen according to \eqref{eq:c_reg} to guarantee a level of at most $\alpha = 0.1$.

The level of the MAP test with a posteriori choice of $\gamma$ is estimated by the maximum of its empirical size over $N_\text{level} = 100$ values of $\udag$ that satisfy $H$ and \cref{ass:spectral}.\ref{ass:source_cond}. The size is computed empirically as the rejection percentage using $M_\text{level} = 500$ data samples for each value of $\udag$. The values $u_1^\dagger, \dots, u_{N_\text{level}}^\dagger$ of $\udag$ are drawn randomly. 
To this end, a discrete source element $\underline{w_n}_N \in \R^N$ is drawn from a uniform distribution on the unit sphere $S^{N-1}$.
Then, the discretization of $u_n^\dagger$ is set as
\[ \underline{u_n^\dagger}_N = \begin{cases}
	\rho\underline{(T^*T)^\frac{\nu}{2}}_N \underline{w_n}_N & \text{if}~\bigscalprod{\underline{\phi}_N,\underline{(T^*T)^\frac{\nu}{2}}_N \underline{w_n}_N}_2 \le 0, \\
	-\rho\underline{(T^*T)^\frac{\nu}{2}}_N \underline{w_n}_N & \text{if}~\bigscalprod{\underline{\phi}_N,\underline{(T^*T)^\frac{\nu}{2}}_N \underline{w_n}_N}_2 > 0,
\end{cases} \]
where $\underline{(T^*T)^\frac{\nu}{2}}_N$: $\R^N \to \R^N$ denotes the discretization of the operator $(T^*T)^\frac{\nu}{2}$.

Power and level are computed for a range of noise levels starting at $\sigma = 1$. Then, the noise level is decreased iteratively by a factor of $0.9$ in case of the power and by $0.9^5$ in case of the level.
The computation of the empirical power and level of the MAP test with a posteriori choice of $\gamma$ is aborted dynamically. For the power, the abort criteria is to exceed a value of $0.99$ over one magnitude of noise levels. For the level, the abort criteria is to stay below a value of $0.01$ over one magnitude of noise levels.

\subsection{Results}

\subsubsection{Deconvolution}

In \cref{fig:prob1,fig:gamma0_prob1}, we depict the results for the deconvolution problem.
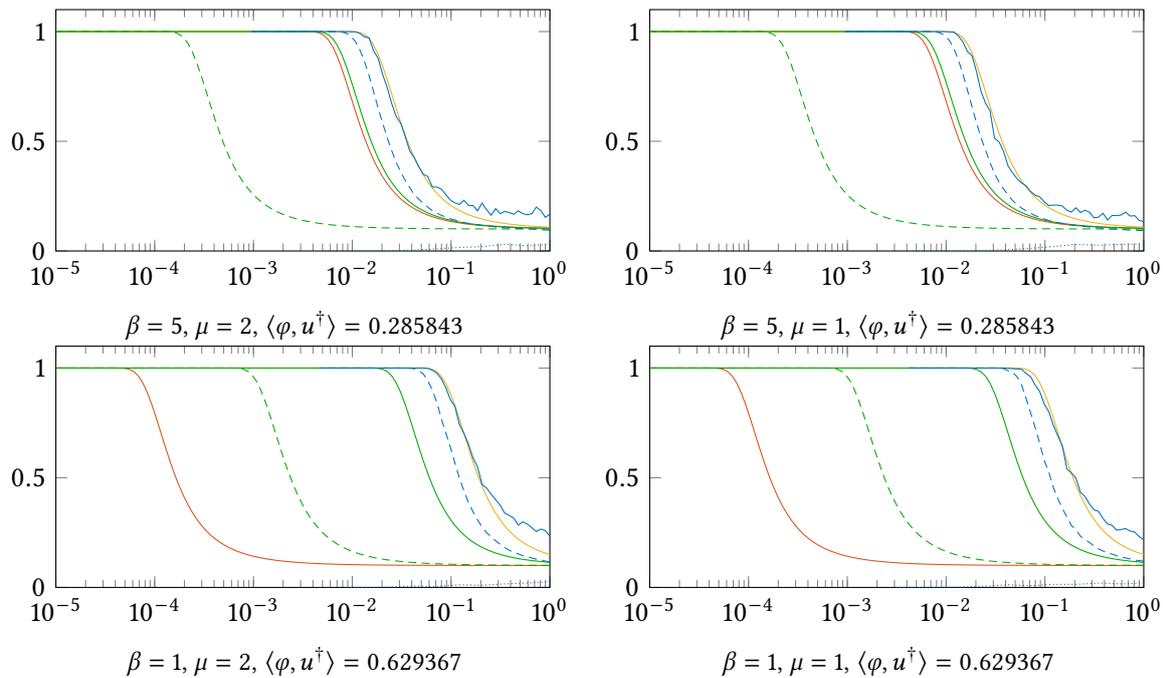
\begin{figure}[!htb]
	\centering
	\subcaptionbox*{$\beta = 5$, $\mu = 2$, $\scalprod{\phi,\udag} = 0.285843$}[0.49\textwidth][l]{\begin{tikzpicture}[baseline]

\begin{axis}[%
	width=\fwidth,
	height=\fheight,
	scale only axis,
	xmode=log,
	xmin=\sigmamindeconv,
	xmax=\sigmamaxdeconv,
	xminorticks=true,
	ymin=0,
	ymax=1.1,
]
\addplot[color=myred] table [x index=0, y index=1] {data/power_prob1_a2_phi5_alpha0.05_omega0.0001_nu1_mu2_unreg.dat}; 
\label{plot:unreg}
\addplot[color=myyellow] table [x index=0, y index=2] {data/power_prob1_a2_phi5_alpha0.05_omega0.0001_nu1_mu2_unreg.dat};
\label{plot:oracle_map}
\addplot[color=mygreen] table [x index=0, y index=3] {data/power_prob1_a2_phi5_alpha0.05_omega0.0001_nu1_mu2_unreg.dat}; 
\label{plot:oracle_pri_map}
\addplot[color=mygreen,densely dashed] table [x index=0, y index=4] {data/power_prob1_a2_phi5_alpha0.05_omega0.0001_nu1_mu2_unreg.dat}; 
\label{plot:est_pri_map}
\addplot[color=myblue,densely dashed] table [x index=0, y index=1] {data/power_prob1_a2_phi5_alpha0.05_omega0.0001_nu1_mu2_M1000.dat};
\label{plot:map_2s}
\addplot[color=myblue] table [x index=0, y index=2] {data/power_prob1_a2_phi5_alpha0.05_omega0.0001_nu1_mu2_M1000.dat};
\label{plot:map_1s}
\addplot[color=myblue,densely dotted] table [x index=0, y index=1] {data/level_prob1_a2_phi5_alpha0.05_omega0.0001_nu1_mu2_M500_N100.dat};
\label{plot:level_map_1s}

\end{axis}

\end{tikzpicture}%
}
	\subcaptionbox*{$\beta = 5$, $\mu = 1$, $\scalprod{\phi,\udag} = 0.285843$}[0.49\textwidth][l]{\begin{tikzpicture}[baseline]

\begin{axis}[%
	width=\fwidth,
	height=\fheight,
	scale only axis,
	xmode=log,
	xmin=\sigmamindeconv,
	xmax=\sigmamaxdeconv,
	xminorticks=true,
	ymin=0,
	ymax=1.1,
]
\addplot[color=myred] table [x index=0, y index=1] {data/power_prob1_a2_phi5_alpha0.05_omega0.0001_nu1_mu1_unreg.dat}; 
\addplot[color=myyellow] table [x index=0, y index=2] {data/power_prob1_a2_phi5_alpha0.05_omega0.0001_nu1_mu1_unreg.dat};
\addplot[color=mygreen] table [x index=0, y index=3] {data/power_prob1_a2_phi5_alpha0.05_omega0.0001_nu1_mu1_unreg.dat}; 
\addplot[color=mygreen,densely dashed] table [x index=0, y index=4] {data/power_prob1_a2_phi5_alpha0.05_omega0.0001_nu1_mu1_unreg.dat}; 
\addplot[color=myblue,densely dashed] table [x index=0, y index=1] {data/power_prob1_a2_phi5_alpha0.05_omega0.0001_nu1_mu1_M1000.dat};
\addplot[color=myblue] table [x index=0, y index=2] {data/power_prob1_a2_phi5_alpha0.05_omega0.0001_nu1_mu1_M1000.dat};
\addplot[color=myblue,densely dotted] table [x index=0, y index=1] {data/level_prob1_a2_phi5_alpha0.05_omega0.0001_nu1_mu1_M500_N100.dat};

\end{axis}

\end{tikzpicture}%
}\\
	\subcaptionbox*{$\beta = 1$, $\mu = 2$, $\scalprod{\phi,\udag} = 0.629367$}[0.49\textwidth][l]{\begin{tikzpicture}[baseline]

\begin{axis}[%
	width=\fwidth,
	height=\fheight,
	scale only axis,
	xmode=log,
	xmin=\sigmamindeconv,
	xmax=\sigmamaxdeconv,
	xminorticks=true,
	ymin=0,
	ymax=1.1,
]
\addplot[color=myred] table [x index=0, y index=1] {data/power_prob1_a2_phi1_alpha0.05_omega0.0001_nu1_mu2_unreg.dat}; 
\addplot[color=myyellow] table [x index=0, y index=2] {data/power_prob1_a2_phi1_alpha0.05_omega0.0001_nu1_mu2_unreg.dat};
\addplot[color=mygreen] table [x index=0, y index=3] {data/power_prob1_a2_phi1_alpha0.05_omega0.0001_nu1_mu2_unreg.dat}; 
\addplot[color=mygreen,densely dashed] table [x index=0, y index=4] {data/power_prob1_a2_phi1_alpha0.05_omega0.0001_nu1_mu2_unreg.dat}; 
\addplot[color=myblue,densely dashed] table [x index=0, y index=1] {data/power_prob1_a2_phi1_alpha0.05_omega0.0001_nu1_mu2_M1000.dat};
\addplot[color=myblue] table [x index=0, y index=2] {data/power_prob1_a2_phi1_alpha0.05_omega0.0001_nu1_mu2_M1000.dat};
\addplot[color=myblue,densely dotted] table [x index=0, y index=1] {data/level_prob1_a2_phi1_alpha0.05_omega0.0001_nu1_mu2_M500_N100.dat};

\end{axis}

\end{tikzpicture}%
}
	\subcaptionbox*{$\beta = 1$, $\mu = 1$, $\scalprod{\phi,\udag} = 0.629367$}[0.49\textwidth][l]{\begin{tikzpicture}[baseline]

\begin{axis}[%
	width=\fwidth,
	height=\fheight,
	scale only axis,
	xmode=log,
	xmin=\sigmamindeconv,
	xmax=\sigmamaxdeconv,
	xminorticks=true,
	ymin=0,
	ymax=1.1,
]
\addplot[color=myred] table [x index=0, y index=1] {data/power_prob1_a2_phi1_alpha0.05_omega0.0001_nu1_mu1_unreg.dat}; 
\addplot[color=myyellow] table [x index=0, y index=2] {data/power_prob1_a2_phi1_alpha0.05_omega0.0001_nu1_mu1_unreg.dat};
\addplot[color=mygreen] table [x index=0, y index=3] {data/power_prob1_a2_phi1_alpha0.05_omega0.0001_nu1_mu1_unreg.dat}; 
\addplot[color=mygreen,densely dashed] table [x index=0, y index=4] {data/power_prob1_a2_phi1_alpha0.05_omega0.0001_nu1_mu1_unreg.dat}; 
\addplot[color=myblue,densely dashed] table [x index=0, y index=1] {data/power_prob1_a2_phi1_alpha0.05_omega0.0001_nu1_mu1_M1000.dat};
\addplot[color=myblue] table [x index=0, y index=2] {data/power_prob1_a2_phi1_alpha0.05_omega0.0001_nu1_mu1_M1000.dat};
\addplot[color=myblue,densely dotted] table [x index=0, y index=1] {data/level_prob1_a2_phi1_alpha0.05_omega0.0001_nu1_mu1_M500_N100.dat};

\end{axis}

\end{tikzpicture}%
}
	\caption{Exact powers of the unregularized test (\ref{plot:unreg}), the oracle MAP test (\ref{plot:oracle_map}), and the oracle a priori MAP test (\ref{plot:oracle_pri_map}), the oracle lower bound for the power of the a priori MAP test (\ref{plot:est_pri_map}), the empirical power of the $2$ sample MAP test (\ref{plot:map_2s}), and the empirical power (\ref{plot:map_1s}) and level (\ref{plot:level_map_1s}) of the $1$ sample MAP test for the deconvolution problem with different values of $\beta$ and $\mu$.}
	\label{fig:prob1}
\end{figure}
%
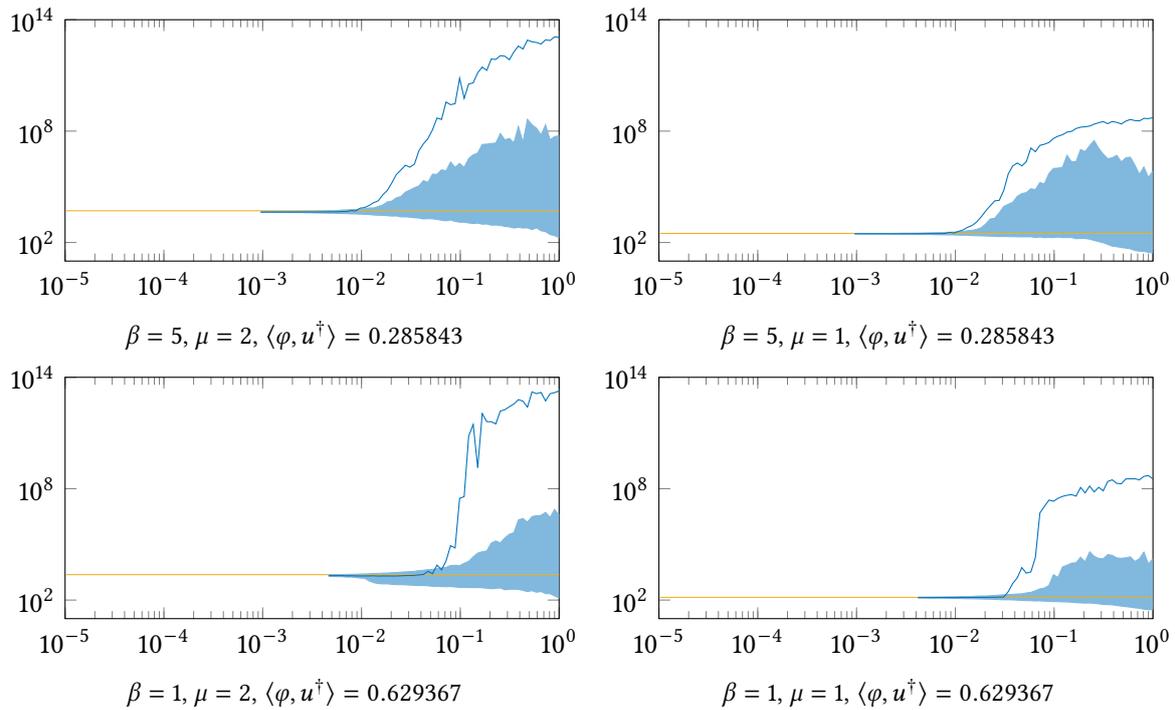
\begin{figure}[!htb]
	\centering
	\subcaptionbox*{$\beta = 5$, $\mu = 2$, $\scalprod{\phi,\udag} = 0.285843$}[0.49\textwidth][l]{\begin{tikzpicture}[baseline]

\begin{axis}[%
	width=\fwidth,
	height=\fheight,
	scale only axis,
	xmode=log,
	xmin=\sigmamindeconv,
	xmax=\sigmamaxdeconv,
	xminorticks=true,
	ymode=log,
	ymin=1e1,
	ymax=1e14,
]
\addplot[color=myyellow] table [x index=0, y index=5] {data/power_prob1_a2_phi5_alpha0.05_omega0.0001_nu1_mu2_unreg.dat};
\addplot[color=mylightblue, name path=A] table [x index=0, y index=4] {data/power_prob1_a2_phi5_alpha0.05_omega0.0001_nu1_mu2_M1000.dat};
\addplot[color=mylightblue, name path=B] table [x index=0, y index=5] {data/power_prob1_a2_phi5_alpha0.05_omega0.0001_nu1_mu2_M1000.dat};
\addplot[color=mylightblue] fill between [of=A and B];
\label{plot:gamma0_quantiles}
\addplot[color=myblue] table [x index=0, y index=3] {data/power_prob1_a2_phi5_alpha0.05_omega0.0001_nu1_mu2_M1000.dat};
\end{axis}

\end{tikzpicture}%
}
	\subcaptionbox*{$\beta = 5$, $\mu = 1$, $\scalprod{\phi,\udag} = 0.285843$}[0.49\textwidth][l]{\begin{tikzpicture}[baseline]

\begin{axis}[%
	width=\fwidth,
	height=\fheight,
	scale only axis,
	xmode=log,
	xmin=\sigmamindeconv,
	xmax=\sigmamaxdeconv,
	xminorticks=true,
	ymode=log,
	ymin=1e1,
	ymax=1e14,
]
\addplot[color=myyellow] table [x index=0, y index=5] {data/power_prob1_a2_phi5_alpha0.05_omega0.0001_nu1_mu1_unreg.dat};
\addplot[color=mylightblue, name path=A] table [x index=0, y index=4] {data/power_prob1_a2_phi5_alpha0.05_omega0.0001_nu1_mu1_M1000.dat};
\addplot[color=mylightblue, name path=B] table [x index=0, y index=5] {data/power_prob1_a2_phi5_alpha0.05_omega0.0001_nu1_mu1_M1000.dat};
\addplot[color=mylightblue] fill between [of=A and B];
\addplot[color=myblue] table [x index=0, y index=3] {data/power_prob1_a2_phi5_alpha0.05_omega0.0001_nu1_mu1_M1000.dat};
\end{axis}

\end{tikzpicture}%
}\\
	\subcaptionbox*{$\beta = 1$, $\mu = 2$, $\scalprod{\phi,\udag} = 0.629367$}[0.49\textwidth][l]{\begin{tikzpicture}[baseline]

\begin{axis}[%
	width=\fwidth,
	height=\fheight,
	scale only axis,
	xmode=log,
	xmin=\sigmamindeconv,
	xmax=\sigmamaxdeconv,
	xminorticks=true,
	ymode=log,
	ymin=1e1,
	ymax=1e14,
]
\addplot[color=myyellow] table [x index=0, y index=5] {data/power_prob1_a2_phi1_alpha0.05_omega0.0001_nu1_mu2_unreg.dat};
\addplot[color=mylightblue, name path=A] table [x index=0, y index=4] {data/power_prob1_a2_phi1_alpha0.05_omega0.0001_nu1_mu2_M1000.dat};
\addplot[color=mylightblue, name path=B] table [x index=0, y index=5] {data/power_prob1_a2_phi1_alpha0.05_omega0.0001_nu1_mu2_M1000.dat};
\addplot[color=mylightblue] fill between [of=A and B];
\addplot[color=myblue] table [x index=0, y index=3] {data/power_prob1_a2_phi1_alpha0.05_omega0.0001_nu1_mu2_M1000.dat};
\end{axis}

\end{tikzpicture}%
}
	\subcaptionbox*{$\beta = 1$, $\mu = 1$, $\scalprod{\phi,\udag} = 0.629367$}[0.49\textwidth][l]{\begin{tikzpicture}[baseline]

\begin{axis}[%
	width=\fwidth,
	height=\fheight,
	scale only axis,
	xmode=log,
	xmin=\sigmamindeconv,
	xmax=\sigmamaxdeconv,
	xminorticks=true,
	ymode=log,
	ymin=1e1,
	ymax=1e14,
]
\addplot[color=myyellow] table [x index=0, y index=5] {data/power_prob1_a2_phi1_alpha0.05_omega0.0001_nu1_mu1_unreg.dat};
\addplot[color=mylightblue, name path=A] table [x index=0, y index=4] {data/power_prob1_a2_phi1_alpha0.05_omega0.0001_nu1_mu1_M1000.dat};
\addplot[color=mylightblue, name path=B] table [x index=0, y index=5] {data/power_prob1_a2_phi1_alpha0.05_omega0.0001_nu1_mu1_M1000.dat};
\addplot[color=mylightblue] fill between [of=A and B];
\addplot[color=myblue] table [x index=0, y index=3] {data/power_prob1_a2_phi1_alpha0.05_omega0.0001_nu1_mu1_M1000.dat};
\end{axis}

\end{tikzpicture}%
}
	\caption{\rev{Choice of $\gamma_0$ by the oracle MAP test (\ref{plot:oracle_map}) as well as mean (\ref{plot:map_1s}), $16 \%$ and $84 \%$ quantiles (\ref{plot:gamma0_quantiles}) of the choice of $\gamma_0$ by the a posteriori MAP test for the deconvolution problem.}}
	\label{fig:gamma0_prob1}
\end{figure}
First of all, we note that the estimate for the level of the 1 sample MAP test with a posteriori choice of $\gamma$ stays below $\alpha_1 = 0.05$ for all considered noise levels and scenarios, and falls to $0$ as the noise level decreases. The latter suggests that the rejection threshold could be chosen more aggresively than in \eqref{eq:c_reg}.

The power of the 1 sample MAP test with a posteriori choice of $\gamma$ is very close to that of the oracle MAP test, and both have the highest power among all considered tests. For large noise levels, the a posteriori choice even outperforms the oracle choice. 
In the compatible scenario, the oracle MAP test with a priori choice of $\gamma$ has about the same power as the unregularized test.
Somewhat surprisingly, the power of the 1 sample MAP test with a posteriori choice of $\gamma$ exceeds that of the 2 sample MAP test with a posteriori choice of $\gamma$, which in turn exceeds that of the oracle MAP test with a priori choice of $\gamma$.
The oracle lower bound for the power of the MAP test with a priori choice of $\gamma$ appears to be very conservative, but still lies considerably above the power of the unregularized test in the incompatible scenario.

The choice of the exponent $\mu$ in the prior covariance has no noticable effect on the power of the resulting MAP tests.
In contrast, the choice of $\varphi$ (parameterized by $\beta)$ has a big influence on the unregularized test as to be expected. In the incompatible scenario ($\beta = 1$), this test does formally not exist (despite the possibility to compute it numerically), and consequently, its performance is the worst. 
We find that the power of all variants of the MAP test increases in the incompatible scenario. We emphasize that from a practical point of view the incompatible scenario is of higher relevance, as it does not weight contributions inside the interval $[c,c+l]$ of interest differently. This becomes e.g. visible by a larger value $\langle \udag,\varphi\rangle$ of the feature of interest, which leads to a (slightly) improved power of the different MAP tests. 

\rev{The mean a posteriori choice of $\gamma_0$ for the MAP test converges quite rapidly toward the optimal oracle choice of $\gamma_0$ as the noise level decreases. For large noise levels, surprisingly, very large values are chosen for $\gamma_0$, corresponding to a very spread out prior. The distribution of the a posteriori choice of $\gamma_0$ is more concentrated than the mean suggests, indicating the existence of outliers with very large values, and it contracts notably as the noise level decreases. Overall, smaller values are chosen for $\gamma_0$ in case of $\mu = 1$ than in case of $\mu = 2$.}

\subsubsection{Differentiation}

The results for the differentiation problem are depicted in \cref{fig:prob2,fig:gamma0_prob2}, respectively.
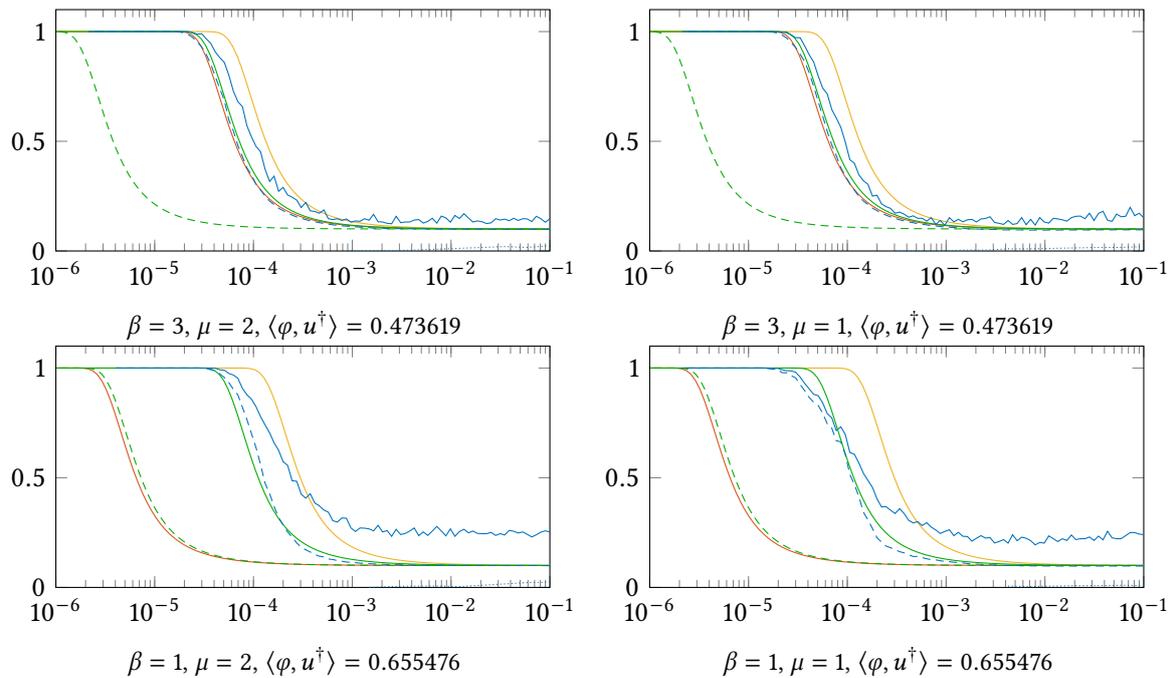
\begin{figure}[!htb]
	\centering
	\subcaptionbox*{$\beta = 3$, $\mu = 2$, $\scalprod{\phi,\udag} = 0.473619$}[0.49\textwidth][l]{\begin{tikzpicture}[baseline]

\begin{axis}[%
	width=\fwidth,
	height=\fheight,
	scale only axis,
	xmode=log,
	xmin=\sigmamindiff,
	xmax=\sigmamaxdiff,
	xminorticks=true,
	ymin=0,
	ymax=1.1,
]
\addplot[color=myred] table [x index=0, y index=1] {data/power_prob2_a2_phi3_alpha0.05_omega1e-10_nu1_mu2_unreg.dat}; 
\addplot[color=myyellow] table [x index=0, y index=2] {data/power_prob2_a2_phi3_alpha0.05_omega1e-10_nu1_mu2_unreg.dat};
\addplot[color=mygreen] table [x index=0, y index=3] {data/power_prob2_a2_phi3_alpha0.05_omega1e-10_nu1_mu2_unreg.dat}; 
\addplot[color=mygreen,densely dashed] table [x index=0, y index=4] {data/power_prob2_a2_phi3_alpha0.05_omega1e-10_nu1_mu2_unreg.dat}; 
\addplot[color=myblue,densely dashed] table [x index=0, y index=1] {data/power_prob2_a2_phi3_alpha0.05_omega1e-10_nu1_mu2_M1000.dat};
\addplot[color=myblue] table [x index=0, y index=2] {data/power_prob2_a2_phi3_alpha0.05_omega1e-10_nu1_mu2_M1000.dat};
\addplot[color=myblue,densely dotted] table [x index=0, y index=1] {data/level_prob2_a2_phi3_alpha0.05_omega1e-10_nu1_mu2_M500_N100.dat};

\end{axis}

\end{tikzpicture}%
}
	\subcaptionbox*{$\beta = 3$, $\mu = 1$, $\scalprod{\phi,\udag} = 0.473619$}[0.49\textwidth][l]{\begin{tikzpicture}[baseline]

\begin{axis}[%
	width=\fwidth,
	height=\fheight,
	scale only axis,
	xmode=log,
	xmin=\sigmamindiff,
	xmax=\sigmamaxdiff,
	xminorticks=true,
	ymin=0,
	ymax=1.1,
]
\addplot[color=myred] table [x index=0, y index=1] {data/power_prob2_a2_phi3_alpha0.05_omega1e-10_nu1_mu1_unreg.dat}; 
\addplot[color=myyellow] table [x index=0, y index=2] {data/power_prob2_a2_phi3_alpha0.05_omega1e-10_nu1_mu1_unreg.dat};
\addplot[color=mygreen] table [x index=0, y index=3] {data/power_prob2_a2_phi3_alpha0.05_omega1e-10_nu1_mu1_unreg.dat}; 
\addplot[color=mygreen,densely dashed] table [x index=0, y index=4] {data/power_prob2_a2_phi3_alpha0.05_omega1e-10_nu1_mu1_unreg.dat}; 
\addplot[color=myblue,densely dashed] table [x index=0, y index=1] {data/power_prob2_a2_phi3_alpha0.05_omega1e-10_nu1_mu1_M1000.dat};
\addplot[color=myblue] table [x index=0, y index=2] {data/power_prob2_a2_phi3_alpha0.05_omega1e-10_nu1_mu1_M1000.dat};
\addplot[color=myblue,densely dotted] table [x index=0, y index=1] {data/level_prob2_a2_phi3_alpha0.05_omega1e-10_nu1_mu1_M500_N100.dat};

\end{axis}

\end{tikzpicture}%
}\\
	\subcaptionbox*{$\beta = 1$, $\mu = 2$, $\scalprod{\phi,\udag} = 0.655476$}[0.49\textwidth][l]{\begin{tikzpicture}[baseline]

\begin{axis}[%
	width=\fwidth,
	height=\fheight,
	scale only axis,
	xmode=log,
	xmin=\sigmamindiff,
	xmax=\sigmamaxdiff,
	xminorticks=true,
	ymin=0,
	ymax=1.1,
]
\addplot[color=myred] table [x index=0, y index=1] {data/power_prob2_a2_phi1_alpha0.05_omega1e-10_nu1_mu2_unreg.dat}; 
\addplot[color=myyellow] table [x index=0, y index=2] {data/power_prob2_a2_phi1_alpha0.05_omega1e-10_nu1_mu2_unreg.dat};
\addplot[color=mygreen] table [x index=0, y index=3] {data/power_prob2_a2_phi1_alpha0.05_omega1e-10_nu1_mu2_unreg.dat}; 
\addplot[color=mygreen,densely dashed] table [x index=0, y index=4] {data/power_prob2_a2_phi1_alpha0.05_omega1e-10_nu1_mu2_unreg.dat}; 
\addplot[color=myblue,densely dashed] table [x index=0, y index=1] {data/power_prob2_a2_phi1_alpha0.05_omega1e-10_nu1_mu2_M1000.dat};
\addplot[color=myblue] table [x index=0, y index=2] {data/power_prob2_a2_phi1_alpha0.05_omega1e-10_nu1_mu2_M1000.dat};
\addplot[color=myblue,densely dotted] table [x index=0, y index=1] {data/level_prob2_a2_phi1_alpha0.05_omega1e-10_nu1_mu2_M500_N100.dat};

\end{axis}

\end{tikzpicture}%
}
	\subcaptionbox*{$\beta = 1$, $\mu = 1$, $\scalprod{\phi,\udag} = 0.655476$}[0.49\textwidth][l]{\begin{tikzpicture}[baseline]

\begin{axis}[%
	width=\fwidth,
	height=\fheight,
	scale only axis,
	xmode=log,
	xmin=\sigmamindiff,
	xmax=\sigmamaxdiff,
	xminorticks=true,
	ymin=0,
	ymax=1.1,
]
\addplot[color=myred] table [x index=0, y index=1] {data/power_prob2_a2_phi1_alpha0.05_omega1e-10_nu1_mu1_unreg.dat}; 
\addplot[color=myyellow] table [x index=0, y index=2] {data/power_prob2_a2_phi1_alpha0.05_omega1e-10_nu1_mu1_unreg.dat};
\addplot[color=mygreen] table [x index=0, y index=3] {data/power_prob2_a2_phi1_alpha0.05_omega1e-10_nu1_mu1_unreg.dat}; 
\addplot[color=mygreen,densely dashed] table [x index=0, y index=4] {data/power_prob2_a2_phi1_alpha0.05_omega1e-10_nu1_mu1_unreg.dat}; 
\addplot[color=myblue,densely dashed] table [x index=0, y index=1] {data/power_prob2_a2_phi1_alpha0.05_omega1e-10_nu1_mu1_M1000.dat};
\addplot[color=myblue] table [x index=0, y index=2] {data/power_prob2_a2_phi1_alpha0.05_omega1e-10_nu1_mu1_M1000.dat};
\addplot[color=myblue,densely dotted] table [x index=0, y index=1] {data/level_prob2_a2_phi1_alpha0.05_omega1e-10_nu1_mu1_M500_N100.dat};

\end{axis}

\end{tikzpicture}%
}
	\caption{Exact powers of the unregularized test (\ref{plot:unreg}), the oracle MAP test (\ref{plot:oracle_map}), and the oracle a priori MAP test (\ref{plot:oracle_pri_map}), the oracle lower bound for the power of the a priori MAP test (\ref{plot:est_pri_map}), the empirical power of the $2$ sample MAP test (\ref{plot:map_2s}), and the empirical power (\ref{plot:map_1s}) and level (\ref{plot:level_map_1s}) of the $1$ sample MAP test for the differentiation problem with different values of $\beta$ and $\mu$.}
	\label{fig:prob2}
\end{figure}
%
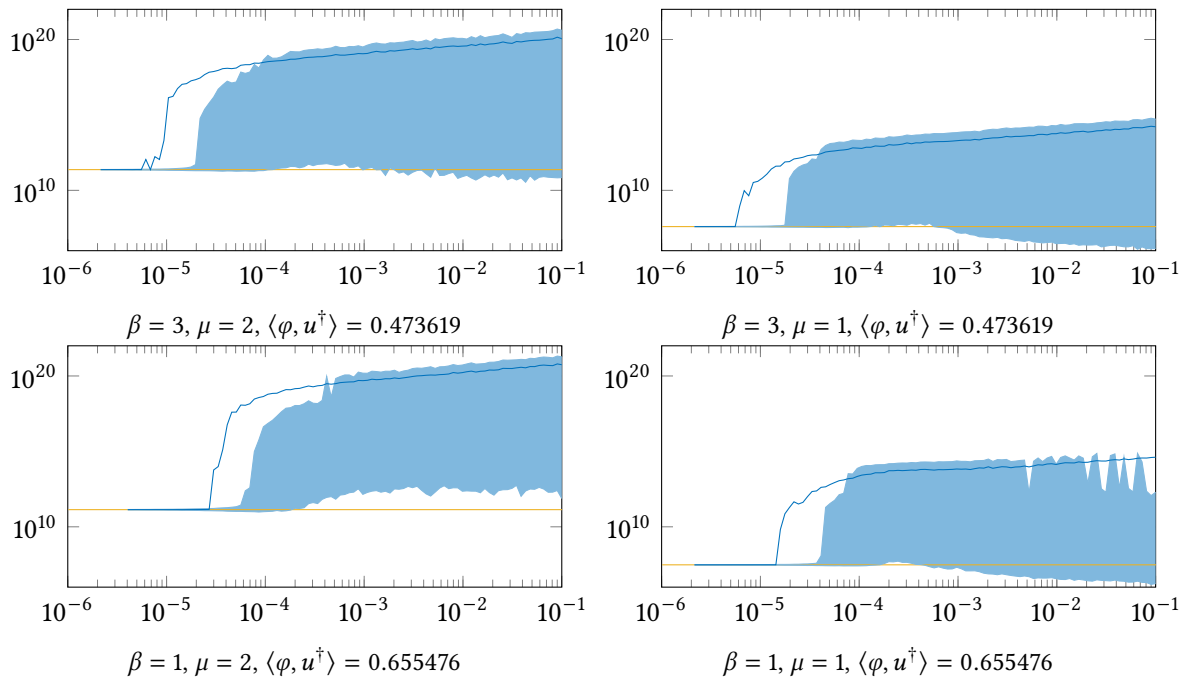
\begin{figure}[!htb]
	\centering
	\subcaptionbox*{$\beta = 3$, $\mu = 2$, $\scalprod{\phi,\udag} = 0.473619$}[0.49\textwidth][l]{\begin{tikzpicture}[baseline]

\begin{axis}[%
	width=\fwidth,
	height=\fheight,
	scale only axis,
	xmode=log,
	xmin=\sigmamindiff,
	xmax=\sigmamaxdiff,
	xminorticks=true,
	ymode=log,
	ymin=1e6,
	ymax=1e22,
]
\addplot[color=myyellow] table [x index=0, y index=5] {data/power_prob2_a2_phi3_alpha0.05_omega1e-10_nu1_mu2_unreg.dat};
\addplot[color=mylightblue, name path=A] table [x index=0, y index=4] {data/power_prob2_a2_phi3_alpha0.05_omega1e-10_nu1_mu2_M1000.dat};
\addplot[color=mylightblue, name path=B] table [x index=0, y index=5] {data/power_prob2_a2_phi3_alpha0.05_omega1e-10_nu1_mu2_M1000.dat};
\addplot[color=mylightblue] fill between [of=A and B];
\addplot[color=myblue] table [x index=0, y index=3] {data/power_prob2_a2_phi3_alpha0.05_omega1e-10_nu1_mu2_M1000.dat};
\end{axis}

\end{tikzpicture}%
}
	\subcaptionbox*{$\beta = 3$, $\mu = 1$, $\scalprod{\phi,\udag} = 0.473619$}[0.49\textwidth][l]{\begin{tikzpicture}[baseline]

\begin{axis}[%
	width=\fwidth,
	height=\fheight,
	scale only axis,
	xmode=log,
	xmin=\sigmamindiff,
	xmax=\sigmamaxdiff,
	xminorticks=true,
	ymode=log,
	ymin=1e6,
	ymax=1e22,
]
\addplot[color=myyellow] table [x index=0, y index=5] {data/power_prob2_a2_phi3_alpha0.05_omega1e-10_nu1_mu1_unreg.dat};
\addplot[color=mylightblue, name path=A] table [x index=0, y index=4] {data/power_prob2_a2_phi3_alpha0.05_omega1e-10_nu1_mu1_M1000.dat};
\addplot[color=mylightblue, name path=B] table [x index=0, y index=5] {data/power_prob2_a2_phi3_alpha0.05_omega1e-10_nu1_mu1_M1000.dat};
\addplot[color=mylightblue] fill between [of=A and B];
\addplot[color=myblue] table [x index=0, y index=3] {data/power_prob2_a2_phi3_alpha0.05_omega1e-10_nu1_mu1_M1000.dat};
\end{axis}

\end{tikzpicture}%
}\\
	\subcaptionbox*{$\beta = 1$, $\mu = 2$, $\scalprod{\phi,\udag} = 0.655476$}[0.49\textwidth][l]{\begin{tikzpicture}[baseline]

\begin{axis}[%
	width=\fwidth,
	height=\fheight,
	scale only axis,
	xmode=log,
	xmin=\sigmamindiff,
	xmax=\sigmamaxdiff,
	xminorticks=true,
	ymode=log,
	ymin=1e6,
	ymax=1e22,
]
\addplot[color=myyellow] table [x index=0, y index=5] {data/power_prob2_a2_phi1_alpha0.05_omega1e-10_nu1_mu2_unreg.dat};
\addplot[color=mylightblue, name path=A] table [x index=0, y index=4] {data/power_prob2_a2_phi1_alpha0.05_omega1e-10_nu1_mu2_M1000.dat};
\addplot[color=mylightblue, name path=B] table [x index=0, y index=5] {data/power_prob2_a2_phi1_alpha0.05_omega1e-10_nu1_mu2_M1000.dat};
\addplot[color=mylightblue] fill between [of=A and B];
\addplot[color=myblue] table [x index=0, y index=3] {data/power_prob2_a2_phi1_alpha0.05_omega1e-10_nu1_mu2_M1000.dat};
\end{axis}

\end{tikzpicture}%
}
	\subcaptionbox*{$\beta = 1$, $\mu = 1$, $\scalprod{\phi,\udag} = 0.655476$}[0.49\textwidth][l]{\begin{tikzpicture}[baseline]

\begin{axis}[%
	width=\fwidth,
	height=\fheight,
	scale only axis,
	xmode=log,
	xmin=\sigmamindiff,
	xmax=\sigmamaxdiff,
	xminorticks=true,
	ymode=log,
	ymin=1e6,
	ymax=1e22,
]
\addplot[color=myyellow] table [x index=0, y index=5] {data/power_prob2_a2_phi1_alpha0.05_omega1e-10_nu1_mu1_unreg.dat};
\addplot[color=mylightblue, name path=A] table [x index=0, y index=4] {data/power_prob2_a2_phi1_alpha0.05_omega1e-10_nu1_mu1_M1000.dat};
\addplot[color=mylightblue, name path=B] table [x index=0, y index=5] {data/power_prob2_a2_phi1_alpha0.05_omega1e-10_nu1_mu1_M1000.dat};
\addplot[color=mylightblue] fill between [of=A and B];
\addplot[color=myblue] table [x index=0, y index=3] {data/power_prob2_a2_phi1_alpha0.05_omega1e-10_nu1_mu1_M1000.dat};
\end{axis}

\end{tikzpicture}%
}
	\caption{\rev{Choice of $\gamma_0$ by the oracle MAP test (\ref{plot:oracle_map}) as well as mean (\ref{plot:map_1s}), $16 \%$ and $84 \%$ quantiles (\ref{plot:gamma0_quantiles}) of the choice of $\gamma_0$ by the a posteriori MAP test for the differentiation problem.}}
	\label{fig:gamma0_prob2}
\end{figure}
The findings are somewhat similar to the deconvolution problem: The unregularized test performs well in the compatible scenario ($\beta = 3$), and nearly fails in the incompatible scenario ($\beta = 1$). The different choice of the exponent $\mu$ in the prior covariance has no visible effect on the results. For this problem, the 1 sample MAP test does not clearly outperform the oracle MAP test, but its performance is still surprisingly good despite the quite fast convergences of its level to $0$. 
\rev{A difference to the convolution problem becomes visible in the choice of $\gamma_0$. First of all, the absolute values are considerably larger for the differentiation problem --- reflecting the required smaller choice of $\omega$ --- and secondly, the values chosen a posteriori converge only slowly to the oracle value. }

\subsubsection{Backward heat equation} 

Finally, the results for the backward heat equation are depicted in \cref{fig:prob3,fig:gamma0_prob3}.
\begin{figure}[!htb]
	\centering
	\subcaptionbox*{$\beta = 1$, $\mu = 2$, $\scalprod{\phi,\udag} = 0.643260$}[0.49\textwidth][l]{\begin{tikzpicture}[baseline]

\begin{axis}[%
	width=\fwidth,
	height=\fheight,
	scale only axis,
	xmode=log,
	xmin=\sigmaminheat,
	xmax=\sigmamaxheat,
	xminorticks=true,
	ymin=0,
	ymax=1.1,
]
\addplot[color=myred] table [x index=0, y index=1] {data/power_prob3_a0.0001_phi1_alpha0.05_omega0.0001_nu1_mu2_unreg.dat}; 
\addplot[color=myyellow] table [x index=0, y index=2] {data/power_prob3_a0.0001_phi1_alpha0.05_omega0.0001_nu1_mu2_unreg.dat};
\addplot[color=mygreen] table [x index=0, y index=3] {data/power_prob3_a0.0001_phi1_alpha0.05_omega0.0001_nu1_mu2_unreg.dat}; 
\addplot[color=mygreen,densely dashed] table [x index=0, y index=4] {data/power_prob3_a0.0001_phi1_alpha0.05_omega0.0001_nu1_mu2_unreg.dat}; 
\addplot[color=myblue,densely dashed] table [x index=0, y index=1] {data/power_prob3_a0.0001_phi1_alpha0.05_omega0.0001_nu1_mu2_M1000.dat};
\addplot[color=myblue] table [x index=0, y index=2] {data/power_prob3_a0.0001_phi1_alpha0.05_omega0.0001_nu1_mu2_M1000.dat};
\addplot[color=myblue,densely dotted] table [x index=0, y index=1] {data/level_prob3_a0.0001_phi1_alpha0.05_omega0.0001_nu1_mu2_M500_N100.dat};

\end{axis}

\end{tikzpicture}%
}
	\subcaptionbox*{$\beta = 1$, $\mu = 1$, $\scalprod{\phi,\udag} = 0.643260$}[0.49\textwidth][l]{\begin{tikzpicture}[baseline]

\begin{axis}[%
	width=\fwidth,
	height=\fheight,
	scale only axis,
	xmode=log,
	xmin=\sigmaminheat,
	xmax=\sigmamaxheat,
	xminorticks=true,
	ymin=0,
	ymax=1.1,
]
\addplot[color=myred] table [x index=0, y index=1] {data/power_prob3_a0.0001_phi1_alpha0.05_omega0.0001_nu1_mu1_unreg.dat}; 
\addplot[color=myyellow] table [x index=0, y index=2] {data/power_prob3_a0.0001_phi1_alpha0.05_omega0.0001_nu1_mu1_unreg.dat};
\addplot[color=mygreen] table [x index=0, y index=3] {data/power_prob3_a0.0001_phi1_alpha0.05_omega0.0001_nu1_mu1_unreg.dat}; 
\addplot[color=mygreen,densely dashed] table [x index=0, y index=4] {data/power_prob3_a0.0001_phi1_alpha0.05_omega0.0001_nu1_mu1_unreg.dat}; 
\addplot[color=myblue,densely dashed] table [x index=0, y index=1] {data/power_prob3_a0.0001_phi1_alpha0.05_omega0.0001_nu1_mu1_M1000.dat};
\addplot[color=myblue] table [x index=0, y index=2] {data/power_prob3_a0.0001_phi1_alpha0.05_omega0.0001_nu1_mu1_M1000.dat};
\addplot[color=myblue,densely dotted] table [x index=0, y index=1] {data/level_prob3_a0.0001_phi1_alpha0.05_omega0.0001_nu1_mu1_M500_N100.dat};

\end{axis}

\end{tikzpicture}%
}
	\caption{Exact powers of the unregularized test (\ref{plot:unreg}), the oracle MAP test (\ref{plot:oracle_map}), and the oracle a priori MAP test (\ref{plot:oracle_pri_map}), the oracle lower bound for the power of the a priori MAP test (\ref{plot:est_pri_map}), the empirical power of the $2$ sample MAP test (\ref{plot:map_2s}), and the empirical power (\ref{plot:map_1s}) and level (\ref{plot:level_map_1s}) of the $1$ sample MAP test for the backward heat equation with $\beta = 1$ and different values of $\mu$.}
	\label{fig:prob3}
\end{figure}
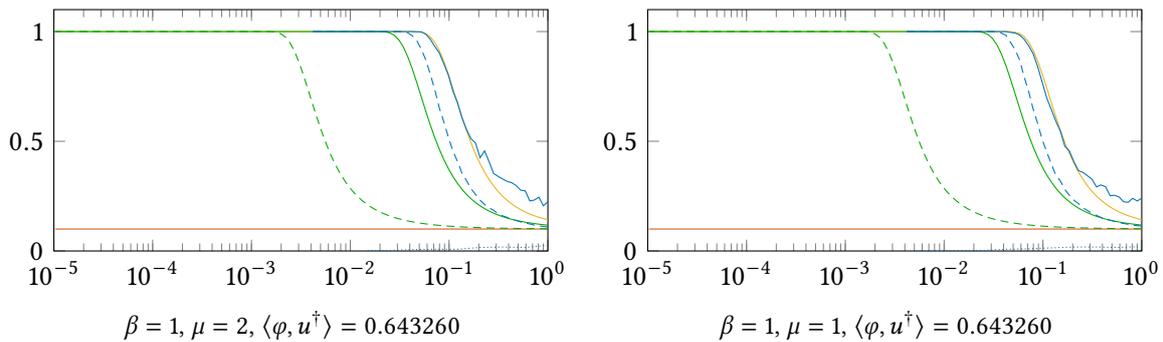
%
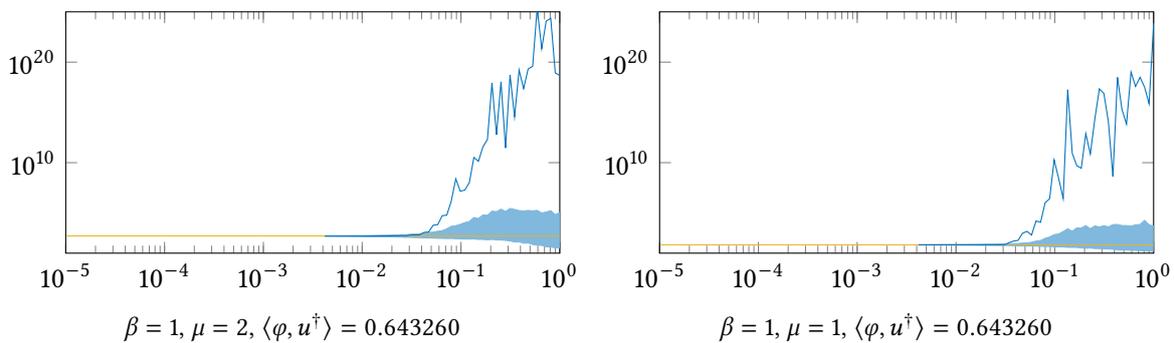
\begin{figure}[!htb]
	\centering
	\subcaptionbox*{$\beta = 1$, $\mu = 2$, $\scalprod{\phi,\udag} = 0.643260$}[0.49\textwidth][l]{\begin{tikzpicture}[baseline]

\begin{axis}[%
	width=\fwidth,
	height=\fheight,
	scale only axis,
	xmode=log,
	xmin=\sigmamindeconv,
	xmax=\sigmamaxdeconv,
	xminorticks=true,
	ymode=log,
	ymin=1e1,
	ymax=1e25,
]
\addplot[color=myyellow] table [x index=0, y index=5] {data/power_prob3_a0.0001_phi1_alpha0.05_omega0.0001_nu1_mu2_unreg.dat};
\addplot[color=mylightblue, name path=A] table [x index=0, y index=4] {data/power_prob3_a0.0001_phi1_alpha0.05_omega0.0001_nu1_mu2_M1000.dat};
\addplot[color=mylightblue, name path=B] table [x index=0, y index=5] {data/power_prob3_a0.0001_phi1_alpha0.05_omega0.0001_nu1_mu2_M1000.dat};
\addplot[color=mylightblue] fill between [of=A and B];
\addplot[color=myblue] table [x index=0, y index=3] {data/power_prob3_a0.0001_phi1_alpha0.05_omega0.0001_nu1_mu2_M1000.dat};
\end{axis}

\end{tikzpicture}%
}
	\subcaptionbox*{$\beta = 1$, $\mu = 1$, $\scalprod{\phi,\udag} = 0.643260$}[0.49\textwidth][l]{\begin{tikzpicture}[baseline]

\begin{axis}[%
	width=\fwidth,
	height=\fheight,
	scale only axis,
	xmode=log,
	xmin=\sigmaminheat,
	xmax=\sigmamaxheat,
	xminorticks=true,
	ymode=log,
	ymin=1e1,
	ymax=1e25,
]
\addplot[color=myyellow] table [x index=0, y index=5] {data/power_prob3_a0.0001_phi1_alpha0.05_omega0.0001_nu1_mu1_unreg.dat};
\addplot[color=mylightblue, name path=A] table [x index=0, y index=4] {data/power_prob3_a0.0001_phi1_alpha0.05_omega0.0001_nu1_mu1_M1000.dat};
\addplot[color=mylightblue, name path=B] table [x index=0, y index=5] {data/power_prob3_a0.0001_phi1_alpha0.05_omega0.0001_nu1_mu1_M1000.dat};
\addplot[color=mylightblue] fill between [of=A and B];
\addplot[color=myblue] table [x index=0, y index=3] {data/power_prob3_a0.0001_phi1_alpha0.05_omega0.0001_nu1_mu1_M1000.dat};
\end{axis}

\end{tikzpicture}%
}
	\caption{\rev{Choice of $\gamma_0$ by the oracle MAP test (\ref{plot:oracle_map}) as well as mean (\ref{plot:map_1s}), $16 \%$ and $84 \%$ quantiles (\ref{plot:gamma0_quantiles}) of the choice of $\gamma_0$ by the a posteriori MAP test for the backward heat equation.}}
	\label{fig:gamma0_prob3}
\end{figure}
Here, it is obvious that a numerical evaluation of the unregularized test fails due to the severe ill-posedness of the problem. Hence, the presented approach of MAP testing enables the investigation of local features for the first time. Similarly to the deconvolution and the differentation problem, the choice of the exponent $\mu$ in the prior covariance does not influence the power of the corresponding tests much. Furthermore, just as in the devonvolution problem, the 1 sample MAP test outperforms the others despite its low level. 
\rev{However, the chosen values of $\gamma_0$ depicted in Figure \ref{fig:gamma0_prob3} reveal some instabilities since the mean is always way larger than the $84\%$ quantile. This implies that there were several runs with extraordinarily large values of $\gamma_0$, but in view of Figure \ref{fig:prob3} those did not seriously influence the power.}


\section{Conclusion and outlook}

In this study, we have described how to test hypotheses of the form \eqref{test} in a Bayesian framework. The derived MAP test has been analyzed from a frequentist's point of view and was characterized as a regularized test in the sense of \cite{KWW:2023}. Furthermore, we have derived a fully data based test by an a posteriori choice of the Gaussian prior. This resolves to some extent the issue of the regularized tests in \cite{KWW:2023} requiring two sets of data. In our numerical simulations, we have finally shown that the investigated approach is widely applicable also to severely ill-posed problems such as the backward heat equation, where no localized testing (in the sense of features $\phi$ with bounded support) has been possible before. One finding is that the 1 sample MAP test seems to perform well or even best in all considered problems and scenarios, and at the same time resolves the two sample issue of the a posteriori test considered in \cite{KWW:2023}.

Future research should be devoted to a better understanding of this surprisingly good performance of the 1 sample MAP test and to a theoretical foundation of its properties. Furthermore, the results shown in our simulations raise the question which feature $\varphi$ should be chosen for a fixed real-world hypothesis (such as \textit{support intersects with $[c,c+l]$}). Since the investigated MAP tests seem to be quite stable w.r.t. different features $\varphi$, it seems interesting to quantify which one yields best results in terms of power. Note that this it not fully possible for the previously used unregularized test, since it exists only for specific features $\varphi$.

Future research should also be devoted to non-Gaussian priors and other a posteriori choices besides the considered one.

\section*{Acknowledgments} 

The authors are supported by the German Research Foundation (DFG) under the grant WE 6204/2-1.
The research of RK has been partially funded by the German Research Foundation (DFG) under project number 318763901 --- SFB 1294.


\bibliography{references}


\appendix

\section{Proofs from Section \ref{sec:map_ssc}}
\label{sec:proofs_map_ssc}

\begin{proof}[Proof of \cref{trace_condition}]
By the $0$-$1$ law \cite[Thm.~2.1.13]{Gine_Nickl:2015}, $\Pi(\V)$ is either $0$ or $1$.
Let $(e_k)_{k\in\N}$ be a system of eigenvectors of $T^*T$ and $(\tau_k^2)_{k\in\N}$ the correspondig eigenvalues.
For $U \sim \Pi$, we have
\[ U = \sum_{k=1}^\infty U_k e_k + m_0 \]
where $U_k \sim \Normal(0,\gamma^2\tau_k^2)$ for all $k \in \N$.
Let us assume that $(T^*T)^{\mu - \nu}$ is trace class and first consider the case that $m_0 = 0$.
Here, we have
\[ \norm{U}_\V^2 = \frac{1}{\rho^2}\norm{(T^*T)^{-\frac{\nu}{2}}U}_\X^2 = \sum_{k=1}^{\infty} \rho^{-2}\tau_k^{-2\nu} U_k^2. \]
Without loss of generality, we consider the case $\gamma = \rho = 1$.
We prove that $\norm{U}_\V$ is almost surely finite for $U \sim \Pi$ using Kolmogorov's three series theorem \cite[Thm.~15.51]{Klenke:2020} with 
\[ A > 0, \quad X_k := \tau_k^{-2\nu} U_k^2, \quad\text{and}\quad Y_k := X_k \mathbf{1}_{\abs{X_k} \le A^2} \quad \text{for all}~k \in \N. \]
First of all, we have
\begin{align*}
	\Prob{\abs{X_k} > A^2} &= \Prob{\tau_k^{-2\nu}U_k^2 > A^2}
	= 2 \Prob{U_k > A\tau_k^\nu}
	= \erfc \left(\frac{A}{\sqrt{2}}\tau_k^{\nu - \mu}\right) \\
	&< \frac{2}{\sqrt{\pi}}\cdot\frac{1}{1 + \frac{A}{\sqrt{2}}\tau_k^{\nu - \mu}} \exp \left(- \frac{A^2}{2}\tau_k^{2\nu - 2\mu}\right)
	\le \frac{2}{\sqrt{\pi}}\exp \left(- \frac{A^2}{2}\tau_k^{2\nu - 2\mu}\right).
\end{align*}
Here, $\erfc$ denotes the complementary error function
\[ \erfc(x) = \frac{2}{\sqrt{\pi}} \int_x^\infty e^{-t} \di t \]
and we used the estimate $M(x) < \frac{1}{x + 1}$, $x \ge 0$, for Mills' ratio 
\[ M(x) = \frac{\int_x^\infty e^{-t^2} \di t}{e^{-x^2}} \]
from \cite{Laforgia_Sismondi:1988}.
Since $t e^{-t} \to 0$ as $t \to \infty$ and $\tau_k \to 0$ as $k \to \infty$, this implies
\[ \sum_{k=1}^\infty \Prob{\abs{X_k} > A^2} < C_1 \sum_{k=1}^\infty \tau_k^{2\mu - 2\nu} = C_1 \tr (T^*T)^{\mu - \nu} < \infty. \]
Second of all, we have
\[ 0 \le \sum_{k=1}^\infty \Exp{Y_k} \le \sum_{k=1}^\infty \Exp{X_k} = \sum_{k=1}^\infty \tau_k^{-2\nu} \Exp{U_k^2} = \sum_{k=1}^\infty \tau_k^{2\mu - 2\nu} = \tr (T^*T)^{\mu - \nu} < \infty. \]
Last of all, we have $\Exp{U_k^4} = 3\tau_k^{4\mu}$ (see, e.g., \cite[p.~148]{Papoulis_Pillai:2002}), which leads to
\begin{align*}
	\Var(Y_k) &\le \Var(X_k) = \Exp{\left(\tau_k^{-2\nu}U_k^2 - \Exp{\tau_k^{-2\nu}U_k^2}\right)^2} \\
	&= \tau_k^{-4\nu}\left(\Exp{U_k^4 - 2U_k^2\Exp{U_k^2} + \Exp{U_k^2}^2}\right) \\
	&= \tau_k^{-4\nu}\left(\Exp{U_k^4} - \Exp{U_k^2}^2\right)
	= \tau_k^{-4\nu} \left(3\tau_k^{4\mu} - \tau_k^{4\mu}\right) = 2 \tau_k^{4\mu - 4\nu}.
\end{align*}
This yields
\[ \sum_{k=1}^\infty \Var{Y_k} \le 2 \sum_{k=1}^\infty \left(\tau_k^{2\mu - 2\nu}\right)^2 \le C_2 \sum_{k=1}^\infty \tau_k^{2\mu - 2\nu} = C_2 \tr (T^*T)^{\mu - \nu} < \infty. \]
Now, $\norm{U}_\V$ is almost surely finite by Kolmogorov's three series theorem.

In the general case $m_0 \in \V$, it now follows that $\norm{U}_\V \le \norm{\sum_{k=1}^\infty U_k e_k}_\V + \norm{m_0}_\V$ is almost surely finite as well.
Now, we can interpret $\Pi$ as a probability measure on $\V$ which satisfies
\begin{align*}
	\int_\V \scalprod{h, x - m_0}_\V \scalprod{k, x - m_0}_\V \Pi(\di x)
	&= \frac{1}{\rho^4} \int_\X \scalprod{(T^*T)^{-\nu}h, x - m_0}_\X \scalprod{(T^*T)^{-\nu}k, x - m_0}_\X \Pi(\di x) \\
	&= \frac{\gamma^2}{\rho^4} \scalprod{(T^*T)^\mu (T^*T)^{-\nu}h, (T^*T)^{-\nu}k}_\X
	= \frac{\gamma^2}{\rho^2} \scalprod{(T^*T)^{\mu - \nu} h, k}_\V
\end{align*}
for all $h, k \in \V$. That is, its covariance w.r.t. the $\V$-norm is given by $\gamma^2\rho^{-2}(T^*T)^{\mu - \nu}$.

On the other hand, if we assume that $\Pi(\V) = 1$ in case that $(T^*T)^{\mu - \nu}$ is not trace class and that $m_0 \in \V$, the previous calculation resulted in a probability measure whose covariance operator is not trace class, which is not possible according to Proposition 1.8 in \cite{DaPrato:2006}. Therefore, $\Pi(\V) = 0$ in this case.

In the remaining case that $(T^*T)^{\mu - \nu}$ is trace class and that $m_0 \notin \V$, we have $\sum_{k=1}^\infty U_k e_k \in \V$ almost surely, which implies that $U = \sum_{k=1}^\infty U_k e_k + m_0 \notin \V$ almost surely.
\end{proof}

\begin{proof}[Proof of \cref{residual_Phi_estimate}]
We have
\begin{align*}
	T^*\Phimap - \phi &= T^*T C_0^\frac12 \left(C_0^\frac12 T^*T C_0^\frac12 + \sigma^2 \Id\right)^{-1} C_0^\frac12 \phi - \phi \\
	&= \left[C_0^\frac12 T^*T C_0^\frac12 \left(C_0^\frac12 T^*T C_0^\frac12 + \sigma^2 \Id\right)^{-1} - \Id\right] \phi \\
	&= -\sigma^2 \left(C_0^\frac12 T^*T C_0^\frac12 + \sigma^2 \Id\right)^{-1} \phi
	= -\left(\Id + \frac{1}{\sigma^2} C_0^\frac12 T^*T C_0^\frac12\right)^{-1} \phi,
\end{align*}
and thus
\begin{align*}
	\bignorm{\left(T^*T\right)^\frac{\nu}{2} \left(T^*\Phimap - \phi\right)}_\X
	&= \bignorm{\left(T^*T\right)^\frac{\nu}{2} \left(\Id + \frac{\gamma^2}{\sigma^2}\left(T^*T\right)^{1 + \mu}\right)^{-1} \phi}_\X \\
	&= \norm{f(T^*T)\phi}_\X \le \norm{f(T^*T)}_{\X \to \X} \norm{\phi}_\X,
\end{align*}
where
\[ f(\lambda) := \frac{\lambda^\frac{\nu}{2}}{1 + \frac{\gamma^2}{\sigma^2}\lambda^{1 + \mu}} \quad \text{for all}~\lambda \ge 0. \]
Moreover, we have
\begin{align*}
	\norm{\Phimap}_\Y &= \bignorm{T\left(T^*T + \sigma^2 C_0^{-1}\right)^{-1}\phi}_\Y
	= \bignorm{(T^*T)^\frac12\frac{1}{\sigma^2}C_0\left(\Id + \frac{1}{\sigma^2}T^*TC_0\right)^{-1}\phi}_\X \\
	&= \bignorm{\frac{\gamma^2}{\sigma^2}(T^*T)^{\frac12 + \mu}\left(\Id + \frac{\gamma^2}{\sigma^2}(T^*T)^{1 + \mu}\right)^{-1}\phi}_\X
	= \norm{h(T^*T)\phi}_\X \le \norm{h(T^*T)}_{\X \to \X} \norm{\phi}_\X
\end{align*}
where
\[ h(\lambda) := \frac{\frac{\gamma^2}{\sigma^2}\lambda^{\frac12 + \mu}}{1 + \frac{\gamma^2}{\sigma^2}\lambda^{1 + \mu}} \quad \text{for all}~\lambda \ge 0. \]
Now, $\norm{f(T^*T)} = \sup_{\lambda \in [0,\kappa]} f(\lambda)$ and $\norm{h(T^*T)} = \sup_{\lambda \in [0,\kappa]} h(\lambda)$, where $\kappa := \norm{T^*T} = \norm{T}^2$.

We find the maximum of $f$ on $[0,\kappa]$ using the derivative
\begin{align*}
	f'(\lambda) &= \frac{\frac{\nu}{2}\lambda^{\frac{\nu}{2} - 1} \left(1 + \frac{\gamma^2}{\sigma^2}\lambda^{1 + \mu}\right) - \lambda^{\frac{\nu}{2}} \frac{\gamma^2}{\sigma^2}(1 + \mu)\lambda^\mu}{\left(\frac{\gamma^2}{\sigma^2}(1 + \mu)\lambda^\mu\right)^2} \\
	&= \left(\frac{\sigma^2}{\gamma^2(1 + \mu)}\right)^2 \lambda^{\frac{\nu}{2} - 1 - 2\mu} \left(\frac{\nu}{2} + \frac{\gamma^2}{\sigma^2}\left(\frac{\nu}{2} - 1 - \mu\right)\lambda^{\mu + 1}\right)
\end{align*}
for $\lambda > 0$. We have
\[ \lim_{\lambda \to 0} f'(\lambda) = \begin{cases}
	0 & \text{if}~\frac{\nu}{2} - 1 > 2\mu, \\
	\frac{\nu}{2}\left(\frac{\sigma^2}{\gamma^2(1 + \mu)}\right)^2 & \text{if}~\frac{\nu}{2} - 1 = 2\mu, \\
	\infty & \text{if}~\frac{\nu}{2} - 1 < 2\mu.
\end{cases} \]
If $\frac{\nu}{2} - 1 \ge \mu$, then $f'(\lambda) > 0$ for all $\lambda > 0$ and $f$ attains its maximum on $[0,\kappa]$ in $\kappa$.
If, on the other hand, $\frac{\nu}{2} - 1 < \mu$, then $f'(\lambda) = 0$ if and only if
\[ \lambda = \left(\frac{\frac{\nu}{2}\sigma^2}{\gamma^2\left(1 + \mu - \frac{\nu}{2}\right)}\right)^{\frac{1}{1 + \mu}} =: \lambda_0(\sigma,\nu,\mu,\gamma), \]
so that $f$ attains its maximum on $[0,\kappa]$ in $\min \{\lambda_0, \kappa\}$.
Then we have
\begin{align*}
	f(\lambda_0) &= \left(\frac{\nu\sigma^2}{2\gamma^2\left(1 + \mu - \frac{\nu}{2}\right)}\right)^{\frac{\nu}{2(1 + \mu)}} \left(1 + \frac{\nu}{2\left(1 + \mu - \frac{\nu}{2}\right)}\right)^{-1} \\
	&= \left(\frac{\gamma^2}{\sigma^2}\left(\frac{2(1 + \mu)}{\nu} - 1\right)\right)^{-\frac{\nu}{2(1 + \mu)}} \left(1 - \frac{\nu}{2(1 + \mu)}\right)
	= g\left(\frac{\nu}{2(1 + \mu)}\right),
\end{align*}
where
\[ g(t) := \left(\frac{\sigma^2 t}{\gamma^2(1 - t)}\right)^t (1 - t) = \left(\frac{\sigma}{\gamma}\right)^{2t} t^t \left(1 - t\right)^{1 - t}. \]
Note that $\frac12 \le t^t(1 - t)^{1 - t} \le 1$ for all $t \in [0,1]$.
For $\mu > \frac{\nu}{2} - 1$ this yields the estimate
\begin{equation*}
	\norm{f(T^*T)} = \sup_{\lambda \in [0,\kappa]} f(\lambda) \le \sup_{\lambda \ge 0} f(\lambda)
	= f(\lambda_0) = g\left(\frac{\nu}{2(1 + \mu)}\right) \le \left(\frac{\sigma}{\gamma}\right)^{\frac{\nu}{1 + \mu}}.
\end{equation*}

Similarly, we find the maximum of $h$ on $[0,\kappa]$ using the derivative
\begin{align*}
	h'(\lambda) &= \frac{\frac{\gamma^2}{\sigma^2}\left(\frac12 + \mu\right)\lambda^{-\frac12 + \mu}\left(1 + \frac{\gamma^2}{\sigma^2}\lambda^{1 + \mu}\right) - \frac{\gamma^2}{\sigma^2}\lambda^{\frac12 + \mu}\frac{\gamma^2}{\sigma^2}(1 + \mu)\lambda^\mu}{\left(\frac{\gamma^2}{\sigma^2}(1 + \mu)\right)^2\lambda^{2\mu}} \\
	&= \frac{\sigma^2\left(\frac12 + \mu\right)}{\gamma^2\left(1 + \mu\right)^2} \lambda^{-\frac12 - \mu} \left(1 - \frac{\gamma^2}{\sigma^2(1 + 2\mu)}\lambda^{1 + \mu}\right).
\end{align*}
Since $\mu > 0$, we have $\lim_{\lambda \to 0} h'(\lambda) = \infty$.
By the choice of $\gamma$, we have $h'(\lambda) = 0$ if and only if
\[ \lambda = \left(\frac{\sigma^2(1 + 2\mu)}{\gamma^2}\right)^\frac{1}{1 + \mu} =: \lambda_1(\sigma,\mu,\gamma) \]
Now, we have
\begin{align*}
	h(\lambda_1) &= \frac{\gamma^2}{\sigma^2} \left(\frac{\sigma^2(1 + 2\mu)}{\gamma^2}\right)^\frac{1 + 2\mu}{2 + 2\mu} \left(1 + (1 + 2\mu)\right)^{-1} \\
	&= \frac{\gamma^2}{\sigma^2} \left(\frac{\sigma^2(1 + 2\mu)}{\gamma^2}\right)^\frac{1 + 2\mu}{2 + 2\mu} \left(\frac{1}{2 + 2\mu}\right)^{\frac{1 + 2\mu}{2 + 2\mu} + \frac{1}{2 + 2\mu}} \\
	&= \frac{\gamma^2}{\sigma^2} \left(\frac{\sigma^2(1 + 2\mu)}{\gamma^2(2 + 2\mu)}\right)^\frac{1 + 2\mu}{2 + 2\mu} \left(1 - \frac{1 + 2\mu}{2 + 2\mu}\right)^{1 - \frac{1 + 2\mu}{2 + 2\mu}} \\
	&= \frac{\gamma^2}{\sigma^2} g\left(\frac{1 + 2\mu}{2 + 2\mu}\right) 
	= \frac{\gamma^2}{\sigma^2} g\left(1 - \frac{1}{2(1 + \mu)}\right).
\end{align*}
This results in the estimate
\begin{equation*}
	\norm{h(T^*T)} = \sup_{\lambda \in [0,\kappa]} h(\lambda) \le \sup_{\lambda \ge 0} h(\lambda) = h(\lambda_1) 
	= \left(\frac{\gamma}{\sigma}\right)^2 g\left(1 - \frac{1}{2(1 + \mu)}\right) \le \left(\frac{\gamma}{\sigma}\right)^{\frac{1}{1 + \mu}}. \qedhere
\end{equation*}
\end{proof}

\section{Properties of periodic convolution operator}

\begin{theorem}
	\label{norm_conv_op}
	The operator $T_{\mathrm{conv},P}$: $L^1(-1,1) \to L^1(-1,1)$ as defined in \cref{sec:deconv} is bounded with $\norm{T_{\mathrm{conv},P}}_{L^2 \to L^2} = 1$.
\end{theorem}
\begin{proof}
Let $\Fper$: $L^2(-P/2,P/2) \to \ell^2(\Z)$ denote the \emph{periodic Fourier transform},
\[ (\Fper f)(k) := \widehat{f}(k) := \frac{1}{P} \int_{-\frac{P}{2}}^{\frac{P}{2}} f(x) \exp\left(-\frac{2\pi i}{P}\scalprod{k,x}{}\right) \di x \]
for all $k \in \Z$.
By the periodic convolution theorem and Poisson's summation formula, we have
\begin{align*}
	\Fper\left(T_{\mathrm{conv},P}\tilde{u}\right)(k) &= \Fper\left(\tilde{h} *_P \tilde{u}\right) = P \Fper h_{\text{per},P}(k) \cdot \Fper\tilde{u}(k) \\
	&= (\Fourier h)\left(\frac{k}{P}\right) \cdot \Fper\tilde{u}(k) = \left(1 + 0.06^2\frac{k^2}{P^2}\right)^{-2} \Fper\tilde{u}(k)
\end{align*}
for all $k \in \Z$.
Now, it follows from the isometry of $P^\frac12\Fper$ that
\begin{align*}
	\norm{T_{\mathrm{conv},P}\tilde{u}}_{L^2}^2 &= \bignorm{\tilde{h} *_P \tilde{u}}_{L^2}^2 = P \bignorm{\Fper\left(\tilde{h} *_P \tilde{u}\right)}_{\ell^2}^2
	= P \bignorm{\left(1 + 0.06^2\frac{k^2}{P^2}\right)^{-2} \Fper\tilde{u}(k)}_{\ell^2}^2 \\
	&\le \sup_{k \in \Z} \left(1 + 0.06^2\frac{k^2}{P^2}\right)^{-4} P \bignorm{\Fper\tilde{u}}_{\ell^2}^2 = \norm{\tilde{u}}_{L^2}^2
\end{align*}
for all $\tilde{u} \in L^2(-1,1)$, with equality for constant functions.
Therefore, $T_{\mathrm{conv},P}$ is bounded and its norm is equal to $1$.
\end{proof}

\section{Second antiderivative as Fredholm operator}
\label{sec:fredholm}

In the case $m = 2$, the antiderivative operator $T_{\mathrm{antider},2}$ as defined in \cref{sec:differentiation} can be expressed as a Fredholm operator 
\begin{equation}
	\label{eq:def_Fredholm}
	[T_{\mathrm{antider},2}u](x) = \int_0^1 k(x,y) u(y) \di y	
\end{equation}
with kernel $k$: $[0,1]^2 \to \R$,
\[ k(x,y) = \min \left\{x(1 - y), (1 - x)y\right\}, \]
see also \cite{hw14}. We verify this by differentiating
\[ [T_{\mathrm{antider},2}u](x) = (1 - x) \int_0^x y u(y) \di y + x \int_x^1 (1 - y) u(y) \di y \]
twice. This yields
\begin{align*}
	[T_{\mathrm{antider},2}u]^\prime(x) &= -\int_0^x y u(y) \di y + (1 - x)x u(x) + \int_x^1 (1 - y) u(y) \di y - x(1 - x) u(x) \\
	&= -\int_0^x y u(y) \di y + \int_x^1 (1 - y) u(y) \di y
\end{align*}
and
\[ [T_{\mathrm{antider},2}u]^\dprime(x) = -x u(x) - (1 - x) u(x) = -u(x) \]
almost everywhere for all $u \in L^2(0,1)$.
The functions $\{\sin(\pi n\cdot)\}_{n \in \N}$ form an orthonormal basis of $L^2(0,1)$, see Proposition 4.5.2 (iv) in \cite{Zeidler:1995}.
Expressing $T_{\mathrm{antider},2}u$ as
\[ [T_{\mathrm{antider},2}u](x) = \sum_{n = 1}^\infty \left(\int_0^1 [Tu](t) \sin(\pi n t) \di t\right) \sin(\pi n x) \]
and differentiating twice results in
\[ [T_{\mathrm{antider},2}u]^\dprime(x) = \sum_{n = 1}^\infty -(\pi n)^2 \left(\int_0^1 [Tu](t) \sin(\pi n t) \di t\right) \sin(\pi n x). \]
On the other hand,
\[ [T_{\mathrm{antider},2}u]^\dprime(x) = -u(x) = -\sum_{n = 1}^\infty \left(\int_0^1 u(t) \sin(\pi n t) \di t\right) \sin(\pi n x), \]
which implies that
\[ \int_0^1 [T_{\mathrm{antider},2}u](t) \sin(\pi n t) \di t = (\pi n)^{-2} \int_0^1 u(t) \sin(\pi n t) \di t \quad \text{for all}~n \in \N. \]
This yields the representation
\[ [T_{\mathrm{antider},2}u](x) = \sum_{n = 1}^\infty (\pi n)^{-2} \left(\int_0^1 u(t) \sin(\pi n t) \di t\right) \sin(\pi n x). \]
A straightforward computation shows that this agrees with $R\widetilde{T}_2E$.

\end{document}